\documentclass[10pt, a4paper, twoside]{amsart}

\usepackage{amsthm}
\usepackage{amsmath}
\usepackage{amssymb}
\usepackage{amscd}

\usepackage[all]{xy}
\usepackage{tikz}
\usepackage{tikz-cd}
\usepackage{comment}

\usepackage[backref=page]{hyperref}
\usepackage{cleveref}
\usepackage{adjustbox}
\usepackage{mathtools}
\usepackage{enumitem}

\hypersetup{colorlinks=true,linkcolor=blue,anchorcolor=blue,citecolor=blue}

\newtheorem{thm}{\bf Theorem}[section]
\newtheorem{prop}[thm]{\bf Proposition}
\newtheorem{lem}[thm]{\bf Lemma}
\newtheorem{cor}[thm]{\bf Corollary}
\newtheorem{q}[thm]{\bf Question}

\newtheorem*{thm*}{\bf Theorem}
\newtheorem*{cor*}{\bf Corollary}

\theoremstyle{definition}

\newtheorem{rem}[thm]{\it Remark}

\newtheorem*{df*}{\bf Definition}
\newtheorem*{dfs*}{\bf Definitions}
\newtheorem*{ack*}{\bf Acknowledgements}

\numberwithin{equation}{section}

\def\P{\mathbb{P}}
\def\C{\mathbb{C}}
\def\F{\mathbb{F}}
\def\Q{\mathbb{Q}}

\def\Z{\mathbb{Z}}

\DeclareMathOperator{\tors}{tors}

\DeclareMathOperator{\Gal}{Gal}
\DeclareMathOperator{\nr}{nr}
\DeclareMathOperator{\cl}{cl}

\DeclareMathOperator{\et}{\textrm{\'Et}}

\DeclareMathOperator{\Gm}{\mathbb{G}_{\operatorname{m}}}

\newcommand{\ov}{\overline}
\newcommand{\on}{\operatorname}
\newcommand{\mc}{\mathcal}

\author{Federico Scavia}
\address{CNRS\\
Institut Galil\'ee\\
	Universit\'e Sorbonne Paris Nord\\
	99 avenue Jean-Baptiste Cl\'ement, 93430\\ 
	Villetaneuse, France}
\email{scavia@math.univ-paris13.fr}

\author{Fumiaki Suzuki}

\address{Institute of Algebraic Geometry\\
Leibniz University Hannover\\
Welfengarten 1, 30167, Hannover\\
Germany
}
\email{suzuki@math.uni-hannover.de}

\title[\tiny Non-algebraic geometrically trivial cohomology classes]{
Non-algebraic geometrically trivial cohomology classes over finite fields
}

\date{September 23, 2024}

\begin{document}

\begin{abstract}
 We give the first examples of smooth projective varieties $X$ over a finite field $\F$ admitting a non-algebraic torsion $\ell$-adic cohomology class
 of degree $4$ which vanishes over $\ov{\F}$. We use them to show that two versions of the integral Tate conjecture over $\F$ are not equivalent to one another and that a fundamental exact sequence of Colliot-Th\'el\`ene and Kahn does not necessarily split. 
 Some of our examples have dimension $4$, and are the first known examples of fourfolds with non-vanishing $H^3_{\on{nr}}(X,\Q_{2}/\Z_{2}(2))$.  
\end{abstract}

\maketitle

\section{Introduction}

Let $\F$ be a finite field, $\ell$ be a prime number invertible in $\F$, 
and $X$ be a smooth projective geometrically connected $\F$-variety. We let $\ov{\F}$ be an algebraic closure of $\F$, $G\coloneqq \Gal(\ov{\F}/\F)$ and $\ov{X}\coloneqq X\times_{\F}\ov{\F}$.
We have the $\ell$-adic cycle maps:
\begin{align}
  CH^2(X)\otimes\Z_{\ell}&\to H^{4}(X,\Z_{\ell}(2))\label{strong-tate},\\
  CH^2(X)\otimes\Z_{\ell}&\to H^{4}(\ov{X},\Z_{\ell}(2))^G\label{medium-tate}.
 \end{align}

From the commutative diagram
\begin{equation}\label{hs-intro}
\adjustbox{scale=0.97,center}{
\begin{tikzcd}
 && CH^2(X)\otimes\Z_{\ell} \arrow[d, "(\ref{strong-tate})"] \arrow[dr, "(\ref{medium-tate})"]  \\
0 \arrow[r] & H^1(\F, H^{3}(\ov{X},\Z_{\ell}(2))) \arrow[r,"\iota"] & H^{4}(X,\Z_{\ell}(2)) \arrow[r] &   H^{4}(\ov{X},\Z_{\ell}(2))^G\arrow[r] & 0,
\end{tikzcd}
}
\end{equation}
where the exact row is induced by the Hochschild-Serre spectral sequence, we see that the surjectivity of (\ref{strong-tate}) implies that of (\ref{medium-tate}). The converse holds in all known examples \cite{antieau2016tate, colliot2010autour, kameko2015tate, pirutka2011groupe, pirutka2015note, quick2011torsion, scavia2022autour}. In particular, the following question had remained open.
\begin{q}\label{are-they-equiv}
    Is it true that, as $X$ ranges over all smooth projective geometrically connected $\F$-varieties, (\ref{strong-tate}) is surjective if and only if (\ref{medium-tate}) is surjective?
\end{q}

In other words: Is it possible for the integral Tate conjecture to fail, (solely) due to the existence of non-algebraic geometrically trivial classes? If so, this would lead to a new interesting obstruction to the integral Tate conjecture over $\F$, which vanishes over $\ov{\F}$. 

Combined with a theorem of Saito \cite[Theorem 8.6]{saito1989observations} (see also \cite[Proposition 3.2]{colliot1999conjectures}), surjectivity of (\ref{strong-tate}) for all three-dimensional $X$ implies Colliot-Th\'el\`ene's conjecture \cite[Conjecture 2.2]{colliot1999conjectures} on zero-cycles of prime-to-$\ell$ degree on smooth projective geometrically connected 
varieties over global function fields of arbitrary dimensions. A positive answer to \cref{are-they-equiv} would be a significant step towards Colliot-Th\'el\`ene's conjecture, reducing surjectivity of (\ref{strong-tate}) to that of (\ref{medium-tate}).

The $\ell$-adic cycle map (\ref{strong-tate}) is related to the third unramified cohomology group
$H^{3}_{\nr}(X,\Q_{\ell}/\Z_{\ell}(2))$, a stable birational invariant of $X$ which is known to be very difficult to compute. According to a theorem of Kahn \cite[Th\'eor\`eme 1.1]{kahn2012classes},
if the group $H^3_{\nr}(X,\Q_\ell/\Z_\ell(2))$ is trivial then the cokernel of (\ref{strong-tate}) is torsion-free, hence (\ref{strong-tate}) is surjective if the Tate conjecture in codimension $2$ is true. Colliot-Th\'el\`ene and Kahn showed that the converse follows from the Beilinson conjecture and a strong form of the Tate conjecture \cite[Th\'eor\`eme 3.18]{colliot2013cycles}. 

It is a challenging problem to construct examples with $H^3_{\nr}(X,\Q_\ell/\Z_\ell(2))\neq 0$ of dimension as small as possible; see e.g. \cite[Question 5.4]{colliot2013cycles} and \cite[\S 5.4]{colliot2023liste}. The known examples of smallest dimension are due to Pirutka \cite{pirutka2011groupe} and have dimension $5$. On the other hand, $H^3_{\nr}(X,\Q_\ell/\Z_\ell(2))$ vanishes if $\dim(X)\leq 2$. The following question naturally arises.

\begin{q}\label{h3nr-question}
Is it true that, for every smooth projective geometrically connected $\F$-variety of dimension $\dim(X)\in \{3,4\}$, we have $H^3_{\nr}(X,\Q_\ell/\Z_\ell(2))=0$?
\end{q}

Codimension $2$ cycles and $H^{3}_{\nr}(X,\Q_{\ell}/\Z_{\ell}(2))$ also appear in the following exact sequence of Colliot-Th\'el\`ene and Kahn \cite[Theorem 6.8]{colliot2013cycles}:
\begin{align}\label{ct-kahn}
    \nonumber 0&\to \on{Ker}(CH^2(X)\to CH^2(\ov{X}))\{\ell\}\xrightarrow{\varphi_{\ell}} H^1(\F, H^3(\ov{X},\Z_{\ell}(2))_{\on{tors}})\\
    &\to \on{Ker}\left(H^3_{\on{nr}}(X,\Q_{\ell}/\Z_{\ell}(2))\to  H^3_{\on{nr}}(\ov{X},\Q_{\ell}/\Z_{\ell}(2))\right) \\ 
    \nonumber &\to\on{Coker}\left(CH^2(X)
    \to CH^2(\ov{X})^G\right)\{\ell\}\to 0,
\end{align}
where the composition of $\varphi_\ell$ with the natural map \[H^1(\F,H^3(\ov{X},\Z_\ell(2))_{\on{tors}})\to H^1(\F,H^3(\ov{X},\Z_\ell(2)))\] coincides with (\ref{strong-tate}).
The following basic question was left open.

\begin{q}\label{phi-ell-surj}
    Is it true that, for every smooth projective geometrically connected $\F$-variety $X$, the map $\varphi_{\ell}$ of (\ref{ct-kahn}) is surjective?
\end{q}

If $\varphi_{\ell}$ were surjective, that is, an isomorphism, the sequence (\ref{ct-kahn}) would split and give even more precise information about the Galois descent of codimension $2$ cycles on $X$. In \cite{colliot2020conjecture}, several equivalent conditions for the surjectivity of $\varphi_{\ell}$ were established, and it was proved that $\varphi_{\ell}$ is surjective if $X$ is the product of a smooth projective surface and an arbitrary number of smooth projective curves.

\subsection{Non-algebraic geometrically trivial cohomology classes}

Our first theorem shows that, despite the above expectations, the answers to \Cref{are-they-equiv} and \Cref{phi-ell-surj} are negative in general.

\begin{thm}\label{mainthm-tate}
    Let $\F$ be a finite field, $\ell$ be a prime number invertible in $\F$, and suppose that $\F$ contains a primitive $\ell^2$-th root of unity. There exists a smooth projective geometrically connected $\F$-variety $X$ of dimension $2\ell+3$ such that:

    (i) the map (\ref{medium-tate}) is surjective  whereas the map  (\ref{strong-tate}) is not, and
    
    (ii) the homomorphism $\varphi_{\ell}$ is not an isomorphism for $X$.
\end{thm}

More precisely, we construct examples where the image of the composition
\begin{equation}\label{composition}H^1(\F, H^3(\ov{X},\Z_\ell(2))_{\on{tors}})\to H^1(\F, H^3(\ov{X},\Z_\ell(2))) \xrightarrow{\iota} H^4(X,\Z_\ell(2))
\end{equation}
contains a class $\alpha$ whose mod $\ell$ reduction is not algebraic. 

\Cref{mainthm-tate} provides new examples of smooth projective varieties $X$ satisfying $H^3_{\nr}(X,\Q_\ell/\Z_\ell(2))\neq 0$, though the construction only works for $\dim (X)\geq 7$. Using ideas from the proof of \cref{mainthm-tate}, we give the first examples of dimension $4$.

\begin{thm}\label{mainthm}
Let $p$ be an odd prime. There exist a finite field $\F$ of characteristic $p$ and a smooth projective geometrically connected fourfold $X$ over $\F$ for which the image of the composition
\[H^1(\F, H^3(\ov{X},\Z_2(2))_{\on{tors}})\to H^1(\F, H^3(\ov{X},\Z_2(2))) \xrightarrow{\iota} H^4(X,\Z_2(2))\]
contains a non-algebraic torsion class.
In particular, $H^3_{\on{nr}}(X,\Q_{2}/\Z_{2}(2))\neq 0$.
\end{thm}

Thus \Cref{h3nr-question} has negative answer in dimension $4$; the question remains open in dimension $3$. It should be noted that, over the complex numbers, examples of threefolds $X$ such that $H^3_{\on{nr}}(X,\Q_{2}/\Z_{2}(2))\neq 0$ abound. Indeed, the Trento counterexamples to the integral Hodge conjecture over $\C$ due to Koll\'ar \cite{kollar1992trento} have non-trivial $H^3_{\on{nr}}$. Conjecturally, no such example can exist over $\ov{\F}$. Indeed, a celebrated theorem of Schoen \cite{schoen1998integral} states that if the rational Tate conjecture for surfaces over finite fields is true, then the integral Tate conjecture for $1$-cycles over $\ov{\F}$ is true in arbitrary dimension.

The example of \cref{mainthm} is based on a construction of Benoist and Ottem \cite[Theorem 5.3]{benoist2021coniveau}. For such $X$, the map (\ref{strong-tate}) is not surjective, but we do not know whether (\ref{medium-tate}) is surjective. Using \'etale cobordism, we show that our methods cannot be used to find an example of dimension $3$; see \Cref{no-dim3}.

\subsection{New ingredients}

We now explain the new ideas that go into the proof of \Cref{mainthm-tate}. Over the complex numbers, Atiyah and Hirzebruch \cite{atiyah1962analytic} proved that all odd-degree Steenrod operations vanish on algebraic classes, and used this to construct the first counterexamples to the integral Hodge conjecture. Their examples are Godeaux-Serre varieties: quotients of smooth complete intersections by the free action of a finite group. As shown by Totaro \cite[Th\'eor\`eme 2.1]{colliot2010autour}, the method of Atiyah--Hirzebruch may be adapted to give counterexamples to the integral Tate conjecture over an arbitrary {\em algebraically closed} field. This method has later been refined by Pirutka--Yagita \cite{pirutka2015note}. Of course, this does not say anything about \Cref{are-they-equiv}, as the cohomology classes of interest all vanish over $\ov{\F}$.

The key new idea for our proof is that {\em one can use Steenrod operations over $\F$  to detect non-algebraic geometrically trivial classes}. More precisely, we are able to prove the following analogue of the Atiyah--Hirzebruch criterion over a finite field.

\begin{thm}[\Cref{odd-vanish}]\label{odd-vanish-intro}
    Let $i\geq 1$ be an integer, $\ell\geq i$ be a prime number, $\F$ be a finite field containing a primitive $\ell^2$-th root of unity $\zeta$ and $X$ be a smooth projective $\F$-variety. If $\alpha\in H^{2i}(X,\Z/\ell)$ is an algebraic class, then all odd-degree Steenrod operations on $H^*(X,\Z/\ell)$ vanish on $\alpha$.
 \end{thm}

(The assumption $\ell\geq i$ already appears over $\ov{\F}$; see \cite[Th\'eor\`eme 2.1]{colliot2010autour}. It is vacuously true when $i=2$, which is the case we are interested in.) The proof of \Cref{odd-vanish-intro}, is much harder than the proof of its counterpart over algebraically closed fields and requires new ideas: the entire \Cref{sec3} is devoted to it. In particular, a key technical ingredient in our argument is the proof of the compatibility of \'etale Steenrod operations over $\F$ and $\ov{\F}$ with the Hochschild-Serre spectral sequence, which we prove by appealing to \'etale homotopy theory to reduce to the analogous (classical) problem for the Serre spectral sequence for Serre fibrations; see \Cref{hochschild-serre-steenrod} and \Cref{hochschild-serre-steenrod-finite-fields}. (For reasons that will be clarified in the next paragraph, we need to work with simplicial schemes as opposed to schemes, but this does not make the proofs harder.)

In order to prove \Cref{mainthm-tate}, it remains to construct examples to which \Cref{odd-vanish-intro} may be applied. Here we follow the general philosophy of Atiyah--Hirzebruch of using suitable smooth projective ``approximations'' of a classifying space $BG$, where $G$ is an algebraic group. Loosely speaking, this reduces the problem of finding a cohomology class not killed by some odd-degree Steenrod operation to a calculation in the group cohomology of $G$. When $G$ is finite, Atiyah and Hirzebruch considered high-dimensional quotients of a general smooth projective complete intersection $Y$ by the free action of $G$. When $G$ is a reductive group, similar constructions have been carried out by Ekedahl \cite{ekedahl2009approximating} and Pirutka--Yagita \cite{pirutka2015note}. (Over the complex numbers, one may take $BG$ to be a topological space. In positive characteristic, one may view $BG$ as an algebraic stack or a simplicial scheme. We follow the second approach.)

Even with \Cref{odd-vanish-intro} at our disposal, none of the counterexamples to the integral Tate conjecture already in the literature allows us to prove \Cref{mainthm-tate}, since in those cases (\ref{medium-tate}) is not surjective. Instead, we consider the projective linear group $G=\on{PGL}_\ell$. The Chow groups and singular cohomology of $B\on{PGL}_n$ are in general quite mysterious. However, when $n=\ell$ is prime they have been computed by Vistoli \cite{vistoli2007cohomology}, who generalized earlier work of Vezzosi \cite{vezzosi2000chow}. It follows that if $X$ is a sufficiently fine smooth projective approximation of $BG$, then $H^4(\ov{X},\Z_\ell(2))\simeq \Z_\ell$ and that (\ref{medium-tate}) is surjective. Combining Vistoli's results with Galois cohomology computations, we show that, when $\F$ contains a primitive $\ell$-th root of unity, the subgroup of geometrically vanishing classes $H^1(\F,H^3(\ov{X},\Z_\ell))\simeq \Z/\ell$ is generated by an $\ell$-torsion class $\alpha$. Finally, we show by topological means that the odd-degree Steenrod operation $\beta\on{P}^1$ does not vanish on the mod $\ell$ reduction of $\alpha$, and we conclude that $\alpha$ is not algebraic by \Cref{odd-vanish-intro}.

\subsection{Structure of the paper} The present paper is structured as follows. \Cref{sec2} presents some preliminaries on \'etale Steenrod operations. In \Cref{sec3}, we prove the compatibility of Steenrod operations with the mod $\ell$ Hochschild spectral sequence (\cref{hochschild-serre-steenrod}), which leads to a criterion for the existence of non-algebraic geometrically trivial $\ell$-adic cohomology classes for varieties over finite fields (\cref{h1-trick}).
\Cref{sec4} is devoted to the proof of Theorems \ref{mainthm-tate} and \ref{mainthm}. 

\subsection{Two coniveau filtrations and Abel-Jacobi maps}

Recently, Benoist--Ottem \cite{benoist2021coniveau} studied the coniveau and strong coniveau filtrations over the complex numbers and gave the first examples where the two filtrations  differ. 

Theorems \ref{mainthm-tate} and \ref{mainthm} provide the first examples where the two coniveau filtrations differ over $\ov{\F}$.
This is a consequence from the following general fact, to be proved in a subsequent paper using Jannsen's Abel-Jacobi map over finite fields: {\em If $X$ is a smooth projective geometrically connected variety over $\F$ and the coniveau and strong coniveau filtrations on $H^3(\ov{X},\Z_\ell(2))$ agree, then $\varphi_\ell$ is an isomorphism.} 

However, simpler examples where the two coniveau filtrations differ may be given following the original approach of Benoist and Ottem. Indeed, we also show that the constructions and the proofs of Benoist and Ottem adapt to an arbitrary algebraically closed field. In particular, we prove a relative version of Wu's Theorem in arbitrary characteristic.

\subsection{Notation}
Let $k$ be a field. We denote by $\ov{k}$ a separable closure of $k$ and $G\coloneqq \Gal(\ov{k}/k)$ the absolute Galois group of $k$. For a continuous $G$-module $M$,  we denote by $H^i(k, M)\coloneqq H^i(G,M)$ the continuous Galois cohomology of $M$.
 When $k$ is a finite field, we denote it by $\F$.

  If $X$ is a $k$-scheme, we define $\ov{X}\coloneqq X\times_k \ov{k}$. A $k$-variety is a separated $k$-scheme of finite type. For a smooth $k$-variety $X$ of pure dimension and an integer $i\geq 0$, we denote by $CH^i(X)$ the Chow group of codimension $i$ cycles on $X$ modulo rational equivalence.
If $\ell$ is a prime invertible in $k$, the notation $H^j(X,\Z_\ell(m))$ will mean the $\ell$-adic continuous \'etale cohomology group of Jannsen \cite{jannsen1988continuous}.
This coincides with the usual \'etale cohomology $\varprojlim_n H^i(X,\mu_{\ell^n}^{\otimes m})$ when the base field $k$ has finite Galois cohomology, which is the case if $k$ is separably closed or a finite field. We write $H^j(X,\Q_{\ell}/\Z_\ell(m))\simeq \varinjlim H^j(X,\mu_{\ell^n}^{\otimes m})$ for the (continuous) \'etale cohomology of the sheaf $\Q_{\ell}/\Z_\ell(m)$. 
We let $H^i_{\nr}(X,\Q_\ell/\Z_\ell(m))$ be the unramified cohomology group.

For an abelian group $A$, an integer $n\geq 1$, and a prime number $\ell$, we denote $A[n]\coloneqq \left\{a\in A\mid na=0\right\}$, $A\{\ell\}$ the subgroup of $\ell$-primary torsion elements of $A$ and
$A_{\tors}$ the subgroup of torsion elements of $A$.

\section{Preliminaries on \'etale Steenrod operations}\label{sec2}

\subsection{\'Etale homotopy type}\label{etale-homotopy-type}

Let $X_{\bullet}=(X_n)_{n\geq 0}$ be a simplicial scheme, such that $X_n$ is  a locally noetherian scheme for every $n\geq 0$. We will only be interested in the cases when $X_{\bullet}$ is the simplicial scheme associated to a scheme $X$, that is, $X_n=X$ for all $n\geq 0$, and all degeneracy and face maps of $X_{\bullet}$ are the identity, or $X_{\bullet}=B\mc{G}$ is the simplicial classifying space of a linear algebraic group $\mc{G}$ over a field $k$, as defined in \cite[Example 1.2]{friedlander1982etale}. If $x\colon  \on{Spec}(\Omega)\to X_0$ is a geometric point, the pair $(X_{\bullet},x)$ is said to be a pointed simplicial scheme.

We let $\on{AbSh}(X_{\bullet})$ be the abelian category of abelian sheaves on the small \'etale site of $X_{\bullet}$. If $A$ is an object of $\on{AbSh}(X_{\bullet})$, we define the cohomology of $A$ as 
\[H^i(X_{\bullet},A)\coloneqq \on{Ext}^i_{\on{AbSh}(X_{\bullet})}(\Z,A),\]
where $\mathbb{Z}$ is the constant sheaf associated to $\mathbb{Z}$; see \cite[Definition 2.3]{friedlander1982etale}. 

We write $\on{HRR}(X_{\bullet})$ for the left directed category of rigid hypercoverings of $X_{\bullet}$; see \cite[Proposition 4.3]{friedlander1973fibrations}. We let $\et(X_{\bullet})$ be Friedlander's \'etale topological type of $X_{\bullet}$; see \cite[Definition 4.4]{friedlander1982etale}. It is the pro-simplicial set (that is, the pro-object in the category of simplicial sets) defined as the functor
\[\et(X_{\bullet})\colon \on{HRR}(X_{\bullet})\to \on{sSet}\]
sending a rigid hypercovering $U_{\bullet\bullet}\to X_{\bullet}$ to $\pi(\Delta(U_{\bullet\bullet}))$, that is, the simplicial set of level-wise connected components of the diagonal simplicial scheme $\Delta(U_{\bullet\bullet})$. If $(X_{\bullet},x)$ is a pointed simplicial scheme, $\et(X_{\bullet})$ is naturally a pointed pro-simplicial space, that is, a pro-object in the category of pointed simplicial sets, which we denote by $\et(X_{\bullet},x)$.

Let $A$ be an abelian group. If $S_{\bullet}$ is a simplicial set, we denote by $\Z(S_{\bullet})$ the free simplicial abelian group on $S_{\bullet}$, and by  $H^*(S_{\bullet},A)$ the singular cohomology groups of $S_{\bullet}$ with coefficients in $A$, that is, the cohomology of the chain complex corresponding to $\on{Hom}(\Z(S_{\bullet}),A)$ via the Dold-Kan correspondence; see \cite[Corollary III.2.3]{goerss2009simplicial}. We also denote by $A$ the abelian sheaf on the small \'etale site of $X$ associated to $A$. Following \cite[Definition 5.1]{friedlander1982etale} we define
\[H^n(\et(X_{\bullet}),A)\coloneqq \on{colim} H^n(\pi(\Delta(U_{\bullet\bullet})),A),\]
where on the right side we consider cohomology of simplicial sets and the colimit is indexed by $\on{HRR}(X_{\bullet})$.
By \cite[Proposition 5.9]{friedlander1982etale}, there is a natural isomorphism
\begin{equation}\label{identification}H^*(X_{\bullet},A)\xrightarrow{\sim} H^*(\et(X_{\bullet}),A).\end{equation}

If $X$ is a $k$-scheme of finite type and $X_{\bullet}$ is the associated simplicial scheme, we have a canonical identification $H^*(X,A)\simeq H^*(X_{\bullet},A)$, and we define $\et(X)\coloneqq \et(X_{\bullet})$. 

\subsection{\'Etale Steenrod operations}\label{define-steenrod}
Steenrod operations in \'etale cohomology have long been known and used. Epstein \cite{epstein1966steenrod} constructed Steenrod operations in great generality, and his definition may  be applied to \'etale cohomology. Instead, we follow the definition of \cite{feng2020etale}, based on \'etale homotopy theory. This approach is closer to construction of Jardine \cite{jardine1989}. 

Let $\ell$ be a prime number. If $S_{\bullet}$ is a simplicial set, $H^*(S_{\bullet},\Z/\ell)$ is endowed with Steenrod operations, and these operations are natural with respect to simplicial maps $S'\to S$. It follows that $H^*(\et(X_{\bullet}),\Z/\ell)$ is endowed with mod $\ell$ Steenrod operations: power operations
\[\on{P}^n\colon H^*(X_{\bullet},\Z/\ell)\to H^{*+2n(\ell-1)}(X_{\bullet},\Z/\ell),\qquad n\geq 0\]
and a reduced Bockstein homomorphism
\[\beta\colon H^*(X_{\bullet},\Z/\ell)\to H^{*+1}(X_{\bullet},\Z/\ell).\]
By definition, $\beta$ is the connecting homomorphism of the cohomology long exact sequence associated to the short exact sequence 
    \[1\to \Z/\ell \to \Z/\ell^2 \to \Z/\ell\to 1\]
    of \'etale sheaves on $X_{\bullet}$. More generally, for every $m\geq 1$, we have a homomorphism 
    \[\beta_m\colon H^*(X_{\bullet},\Z/\ell)\to H^{*+1}(X_{\bullet},\Z/\ell^m)\]
    obtained as the connecting homomorphism of the cohomology long exact sequence associated to the short exact sequence 
    \[1\to \Z/\ell^m \to \Z/\ell^{m+1} \to \Z/\ell\to 1.\]
    The homomorphisms $\beta_m$ form an inverse system in $m$. Passing to the inverse limit in $m$ yields the non-reduced Bockstein homomorphism
    \[\tilde{\beta}\colon H^*(X_{\bullet},\Z/\ell)\to H^{*+1}(X_{\bullet},\Z_{\ell}).\]
    Letting \[\pi_{\ell}\colon H^*(X_{\bullet},\Z_\ell)\to H^{*}(X_{\bullet},\Z/{\ell})\]
    be the homomorphism of reduction modulo $\ell$, we have \[\beta=\pi_{\ell}\circ\tilde{\beta}.\]
    Since the \'etale Steenrod operations are defined as colimit of the simplicial Steenrod operations, the standard properties of the simplicial operations are also true for the \'etale operations. Therefore, the \'etale Steenrod operations are natural with respect to morphisms of simplicial schemes, $\on{P}^0$ is the identity, $\on{P}^n(x)=x^\ell$ if $|x|=2n$ and $\on{P}^n(x)=0$ if $|x|<2n$, Cartan's formula and the Adem relations hold. 
    
    If $\ell=2$, we define the Steenrod squares as
    \[\on{Sq}^{2n}\coloneqq \on{P}^n,\qquad \on{Sq}^{2n+1}\coloneqq \beta\on{P}^n.\] 
    Since $\on{P}^0$ is the identity, we have $\beta=\on{Sq}^1$. 
    
    We write $\on{P}$ for the total Steenrod $\ell$-th power operation, and $\on{Sq}$ for the total Steenrod square:
    \[\on{P}\coloneqq \sum_{j=0}^{\infty}\on{P}^j,\qquad \on{Sq}\coloneqq \sum_{j=0}^{\infty}\on{Sq}^j.\]
    Cartan's formula translates to the fact that $\on{P}$ and $\on{Sq}$ are ring homomorphisms.

 \subsection{Specialization map}
 
 For the proof of \Cref{mainthm}, it will be useful to know that the specialization map in \'etale cohomology is compatible with Steenrod operations.
 
\begin{lem}\label{smooth-proper}
Let $R$ be a strictly henselian discrete valuation ring, $k$ be the residue field of $R$, $K$ be the fraction field of $R$, $\mc{X}$ be a smooth projective $R$-scheme of finite type, and $\ell$ be a prime number invertible in $k$. 
Fix geometric points
\[
\ov{s}\colon \on{Spec}(k)\to \on{Spec}(R), \qquad \ov{\eta}\colon \on{Spec}(\ov{K})\to \on{Spec}(R).
\]
Then the specialization map in \'etale cohomology \[\on{sp}_{\ov{\eta},\ov{s}}\colon H^*(\mc{X}_{\ov{\eta}},\Z/{\ell})\xrightarrow{\sim} H^*(\mc{X}_{\ov{s}},\Z/{\ell})\]
is compatible with Steenrod operations.
\end{lem}

\begin{proof}
Since $R$ is strictly henselian and $\mc{X}$ is smooth over $R$, there exists a section $x\colon \on{Spec}(R)\to \mc{X}$. We set $x_{\ov{s}} \coloneqq x\circ \ov{s}\in \mc{X}(k)$ and $x_{\ov{\eta}}\coloneqq x\circ \ov{\eta}\in \mc{X}(\ov{K})$. By \cite[Propositions 8.6 and 8.7]{friedlander1982etale}, the closed embedding $\mc{X}_k\hookrightarrow \mc{X}$ and the open embedding $\mc{X}_{K}\hookrightarrow \mc{X}$ induce homotopy equivalences
    \[\on{holim}(\Z/\ell)_{\infty}(\textrm{\'Et}(\mc{X}_{\ov{s}},x_{\ov{s}}))\xrightarrow{i}\on{holim}(\Z/\ell)_{\infty}(\textrm{\'Et}(\mc{X},x_{\ov{s}}))\xleftarrow{j}\on{holim}(\Z/\ell)_{\infty}(\textrm{\'Et}(\mc{X}_{\ov{\eta}},x_{\ov{\eta}})),\]
    where $(\Z/\ell)_{\infty}(-)$ is the Bousfield-Kan pro-$\ell$ completion functor, and $\on{holim}(-)$ is the Bousfield-Kan homotopy (inverse) limit functor; see \cite[p. 57]{friedlander1982etale}. 
    (This is a consequence of the smooth and proper base change theorems in \'etale cohomology.) The specialization map $\on{sp}_{\ov{\eta},\ov{s}}$ is induced by the homotopy equivalence of pointed pro-simplicial sets $j^{-1}\circ i$, hence it commutes with Steenrod operations. 
\end{proof}

There is a similar result for the classifying space of a linear algebraic group. We will need it during the proof of \Cref{mainthm-tate}.

\begin{lem}\label{base-change-bg}
 Let $\mc{G}$ be a reductive group over $\Z$, $k$ an algebraically closed field of positive characteristic, $\ell$ be a prime invertible in $k$. Then we have a commutative square of ring homomorphisms
 \[
 \begin{tikzcd}
 H^*(B_k\mc{G},\Z_{\ell}) \arrow[r,"\sim"] \arrow[d, "\pi_{\ell}"] &   H^*_{\on{sing}}(B\mc{G}(\C),\Z)\otimes\Z_{\ell} \arrow[d, "\pi_{\ell}"] \\
 H^*(B_k\mc{G},\Z/{\ell}) \arrow[r,"\sim"] & H^*_{\on{sing}}(B\mc{G}(\C),\Z/\ell),
 \end{tikzcd}
 \]
 where the horizontal maps are isomorphisms. Moreover, the bottom horizontal map is compatible with \'etale Steenrod operations (on the left side) and topological Steenrod operations (on the right side).
\end{lem}
 
 \begin{proof}
     This follows from the homotopy equivalence
     \[\on{holim}(\Z/\ell)_{\infty}(\textrm{\'Et}(B_k\mc{G}))\xrightarrow{\sim} \on{holim}(\Z/\ell)_{\infty}(\on{Sing}(B\mc{G}(\C)),\]
     established in \cite[Proposition 8.8]{friedlander1982etale}, where $\on{Sing}(-)$ denotes the singular complex.
 \end{proof}

\section{Non-algebraicity criterion}\label{sec3}

The purpose of this section is the proof of \Cref{h1-trick}, which provides a criterion for the existence of non-algebraic geometrically trivial $\ell$-adic cohomology classes for smooth projective varieties over finite fields. The key technical ingredient for the proof of \Cref{h1-trick} comes from \Cref{hochschild-serre-steenrod}, which clarifies the behavior of Steenrod operations in the mod $\ell$ Hochschild-Serre spectral sequence. In the case of a finite field $\F$, \Cref{hochschild-serre-steenrod} takes a particularly simple form; see \Cref{hochschild-serre-steenrod-finite-fields}. The reader willing to take on faith the compatibility of the Steenrod operations with the Hochschild-Serre spectral sequence should skip to the statement of \Cref{hochschild-serre-steenrod-finite-fields}.

\subsection{Steenrod operations and the Hochschild-Serre spectral sequence}
Let $k$ be a field, $\ell$ be a prime number, let $X_{\bullet}=(X_n)_{n\geq 0}$ be a simplicial scheme, and suppose that $X_n$ is a $k$-scheme of finite type for all $n\geq 0$. We write $\ov{X}_{\bullet}\coloneqq X_{\bullet}\times_k\ov{k}$. The discussion of \Cref{define-steenrod} yields Steenrod operations on $H^*(X_{\bullet},\Z/\ell)$ and on $H^*(\ov{X}_{\bullet},\Z/\ell)$. 
The two cohomology rings are related by the Hochschild-Serre spectral sequence
    \begin{equation}\label{ss-hochschild-serre-limit}
    E_2^{i,j}\coloneqq H^i(k, H^j(\ov{X}_{\bullet},\Z/\ell))\Rightarrow H^{i+j}(X_{\bullet},\Z/\ell).
    \end{equation}
This is the Grothendieck spectral sequence for the composition of the functors of Galois-equivariant global sections $H^0(\ov{X}_{\bullet},-)$ and of $G$-invariants. We may also view (\ref{ss-hochschild-serre-limit}) as the direct limit of the spectral sequences
    \begin{equation}\label{ss-hochschild-serre}
    E_2^{i,j}\coloneqq H^i(\on{Gal}(K/k), H^j((X_{\bullet})_K,\Z/\ell))\Rightarrow H^{i+j}(X_{\bullet},\Z/\ell)
    \end{equation}
defined in a similar way, where $K/k$ ranges over all finite Galois subextensions of $\ov{k}/k$.

For the proof of Theorems \ref{mainthm-tate} and \ref{mainthm}, it will be important to understand the compatibility of Steenrod operations with (\ref{ss-hochschild-serre-limit}).

Let $\Gamma$ be a finite group. If $S_{\bullet}$ is a simplicial set with a simplicial $\Gamma$-action, we denote by $(S_{\bullet})_{h\Gamma}$ the homotopy quotient of $S_{\bullet}$ by $\Gamma$, namely
    \[(S_{\bullet})_{h\Gamma}\coloneqq (S_{\bullet}\times E\Gamma)/\Gamma,\]
    where $E\Gamma$ is a weakly contractible simplicial set with a free $\Gamma$-action. The projection $(S_{\bullet}\times E\Gamma)\to (S_{\bullet})_{h\Gamma}$ is a principal $\Gamma$-fibration in the sense of \cite[Definition 5.3]{friedlander1982etale}. By construction, we have a fibration of simplicial sets
    \[S_{\bullet}\to (S_{\bullet})_{h\Gamma}\to B\Gamma,\]
    where $B\Gamma\coloneqq E\Gamma/\Gamma$ is a model for the simplicial classifying space for $\Gamma$.  
    We have the Serre spectral sequence
    \begin{equation}\label{ss-hom-quot-bgamma}E_2^{i,j}\coloneqq H^i(B\Gamma, H^j(S_{\bullet},\Z/\ell))\Rightarrow H^{i+j}((S_{\bullet})_{h\Gamma},\Z/\ell),\end{equation}
    where $H^j(S_{\bullet},\Z/\ell)$ is viewed as a local system over $B\Gamma$,
    or equivalently the Grothendieck spectral sequence
    \begin{equation}\label{ss-hom-quot}E_2^{i,j}\coloneqq H^i(\Gamma, H^j(S_{\bullet},\Z/\ell))\Rightarrow H^{i+j}((S_{\bullet})_{h\Gamma},\Z/\ell)\end{equation}
    where $H^i(\Gamma, -)$ denotes group cohomology and $H^j(S_{\bullet},\Z/\ell)$ is now viewed as a $\Gamma$-representation. 
    
    Let $I$ be a left directed category and $\mathcal{S}_{\bullet}=((S_i)_{\bullet})_{i\in I}$ is a pro-simplicial set with a $\Gamma$-action: $\Gamma$ acts on every $(S_i)_{\bullet}$ via simplicial automorphisms and the transition morphisms $(S_i)_{\bullet}\to (S_j)_{\bullet}$ are $\Gamma$-equivariant. We define the homotopy quotient of $\mathcal{S}_{\bullet}$ by the $\Gamma$-action as \[(\mathcal{S}_{\bullet})_{h\Gamma}\coloneqq ((S_i)_{\bullet})_{h\Gamma})_{i\in I},\]
    that is, the pro-simplicial set associated to the inverse system of homotopy quotients.
    
    By definition, $H^*(\mathcal{S}_{\bullet},\Z/\ell)$ is the direct limit of $H^j((S_i)_{\bullet},\Z/\ell)$ over $i\in I$, and similarly for $(\mathcal{S}_{\bullet})_{h\Gamma}$. Since group cohomology commutes with direct limits, by passing to the direct limit in (\ref{ss-hom-quot}) we obtain a spectral sequence
    \begin{equation}\label{ss-pro-hom-quot}
    E_2^{i,j}\coloneqq H^i(\Gamma, H^j(\mathcal{S}_{\bullet},\Z/\ell))\Rightarrow H^{i+j}((\mathcal{S}_{\bullet})_{h\Gamma},\Z/\ell).
    \end{equation}

    \begin{lem}\label{cofinal}
The rigid \'etale hypercoverings of $(X_{\bullet})_K$ that are defined over $k$ are cofinal in $\on{HRR}((X_{\bullet})_K)$. In particular, we have a canonical $\on{Gal}(K/k)$-equivariant isomorphism
\[H^*(\et((X_{\bullet})_K),A)\simeq \on{colim} H^*(\pi(\Delta((U_{\bullet\bullet})_K)),A),\]
where the colimit is over $\on{HRR}(X_{\bullet})$.
\end{lem}

\begin{proof}
Let $\mc{C}$ be the site of \'etale coverings of $(X_{\bullet})_K$ by simplicial schemes, and $\mc{B}$ be the collection of all rigid coverings $Y_{\bullet}\to (X_{\bullet})_K$ that are defined over $k$. 
If $Y_{\bullet}\to (X_{\bullet})_K$ is a rigid covering, the level-wise Galois closure $\tilde{Y}_{\bullet}\to Y_{\bullet}$ is a rigid covering of $Y_{\bullet}$ defined over $k$. Moreover, if $\{(Y_j)_{\bullet}\to Y_{\bullet}\}$ is an \'etale covering of $Y_{\bullet}$ and $Z_{\bullet}\to Y_{\bullet}$ is a rigid covering, then $\{(Y_j)_{\bullet}\to Y_{\bullet}\}\coprod \{Z_{\bullet}\to Y_{\bullet}\}$ is also an \'etale covering of $Y_{\bullet}$. Therefore, the pair $(\mc{C},\mc{B})$ satisfies the assumptions of \cite[0DAV]{stacks-project}. The conclusion now follows from \cite[0DAV]{stacks-project}.
\end{proof}

Let $K/k$ be a finite Galois extension, and let $f\colon (X_{\bullet})_K\to X_{\bullet}$ be the projection map. The group $\on{Gal}(K/k)$ acts on $(X_{\bullet})_K$ via its action on $\on{Spec}(K)$. By \Cref{cofinal}, we may write $\textrm{\'Et}((X_{\bullet})_K)$ as the pro-simplicial set associated to an inverse system of simplicial sets endowed with compatible $\on{Gal}(K/k)$-actions. Therefore, we may let $\mathcal{S}_{\bullet}=\textrm{\'Et}((X_{\bullet})_K)$ and $\Gamma=\on{Gal}(K/k)$ in the previous discussion. The spectral sequence (\ref{ss-pro-hom-quot}) specializes to:
    \begin{equation}\label{ss-etale-hom-quot}
    E_2^{i,j}\coloneqq H^i(\on{Gal}(K/k), H^j(\textrm{\'Et}((X_{\bullet})_K),\Z/\ell))\Rightarrow H^{i+j}(\textrm{\'Et}((X_{\bullet})_K)_{h\on{Gal}(K/k)},\Z/\ell).
    \end{equation}
    
    \begin{lem}\label{map-of-ss}
    For every finite Galois extension $K/k$, there exists an isomorphism between the spectral sequences (\ref{ss-hochschild-serre}) and (\ref{ss-etale-hom-quot}), which is functorial with respect to $K$. 
    \end{lem}

    \begin{proof}
    If $A_{\bullet\bullet}$ is a bisimplicial abelian group, we denote by $M(A_{\bullet\bullet})$ the Moore chain bicomplex corresponding to $A_{\bullet\bullet}$ (see \cite[p. 205]{goerss2009simplicial}), and by $\on{Tot}(M(A_{\bullet\bullet}))$ the associated total complex. (In \cite{goerss2009simplicial}, the Moore complex of $A_{\bullet\bullet}$ is also denoted by $A_{\bullet\bullet}$.) By definition, $H^*(A_{\bullet\bullet})$ is the cohomology of the chain complex $\on{Tot}(M(A_{\bullet\bullet}))$; this is in accordance with \cite[p. 20]{friedlander1982etale}. (Recall that the cohomology of a chain complex is computed by first applying $\on{Hom}(-,A)$ in each degree, and then taking cohomology of the resulting cochain complex.) Since the total complex and Moore bicomplex functors are exact, this defines a $\delta$-functor on the category of bisimplicial abelian groups. If $A^{\bullet\bullet}$ is a bicosimplicial abelian group, we define dually the Moore cochain bicomplex $M(A^{\bullet\bullet})$, and let  $H^*(A^{\bullet\bullet})$ be the cohomology of the total complex of $M(A^{\bullet\bullet})$.
    
    If $F$ is an \'etale sheaf of abelian groups on a bisimplicial scheme $Y_{\bullet\bullet}$, we denote by $F(Y_{\bullet\bullet})$ the bicosimplicial abelian group obtained by taking global sections level-wise.  By \cite[Theorem 3.8]{friedlander1982etale}, we have an isomorphism of $\delta$-functors on $\on{AbSh}(X_{\bullet})$:
    \[H^*((X_{\bullet})_K,-)\xrightarrow{\sim} \on{colim} H^*((-)((U_{\bullet\bullet})_K)),\]
    where the colimit is over all \'etale hypercoverings $U_{\bullet\bullet}\to X_{\bullet}$. (The right hand side is a $\delta$-functor because the colimit functor is exact.) Since group cohomology commutes with colimits, this induces an isomorphism of spectral sequences from (\ref{ss-hochschild-serre}) to the colimit of the spectral sequences
    \begin{equation}\label{ss1}
    E_2^{i,j}\coloneqq H^i(\on{Gal}(K/k), H^j((\Z/\ell)((U_{\bullet\bullet})_K)))\Rightarrow H^{i+j}((\Z/\ell)(U_{\bullet\bullet})),
    \end{equation}
    as $U_{\bullet\bullet}$ ranges over all \'etale hypercoverings of $X_{\bullet}$. The spectral sequences (\ref{ss1}) are defined as the Grothendieck spectral sequences for the composition of $(-)((U_{\bullet\bullet})_K)$ and of the functor of  $\on{Gal}(K/k)$-invariants. 

    Let $A_{\bullet\bullet}$ be a bisimplicial abelian group, $\Delta(A_{\bullet\bullet})$ be the diagonal simplicial abelian group, and $d(A_{\bullet\bullet})$ be the chain complex corresponding to the simplicial diagonal $\Delta(A_{\bullet\bullet})$ via the Dold-Kan correspondence. (In \cite[p. 205]{goerss2009simplicial}, both notions are denoted by $d(A_{\bullet\bullet})$.) By the bisimplicial Eilenberg-Zilber Theorem \cite[Theorem IV.2.3]{goerss2009simplicial}, proved by Dold and Puppe, we have a natural chain homotopy equivalence $\on{Tot}(M(A_{\bullet\bullet}))\simeq d(A_{\bullet\bullet})$. In particular, if $S_{\bullet\bullet}$ is a bisimplicial set, letting $A_{\bullet\bullet}=\Z(S_{\bullet\bullet})$ we get a natural chain homotopy equivalence 
    \begin{equation}\label{dold-puppe}
    \on{Tot}(M(\Z(S_{\bullet\bullet})))\simeq d(\Z(S_{\bullet\bullet}))=\Z(\Delta(S_{\bullet\bullet})).\end{equation}
     Let $U_{\bullet\bullet}\to X_{\bullet}$ by an \'etale hypercovering, $A$ an abelian group. We also denote by $A$ the corresponding constant \'etale sheaf on $X_{\bullet}$ and $U_{\bullet\bullet}$. Then \[A(\Delta((U_{\bullet\bullet})_K))=\on{Hom}(\Z(\pi(\Delta((U_{\bullet\bullet})_K))),A),\] hence by (\ref{dold-puppe}) we obtain an isomorphism of $\delta$-functors 
    \[H^*((-)((U_{\bullet\bullet})_K))\xrightarrow{\sim}H^*(\pi(\Delta((U_{\bullet\bullet})_K)),-)\]
    from the abelian category of constant abelian groups (or constant abelian \'etale sheaves on $U_{\bullet\bullet}$) to the abelian category of $\on{Gal}(K/k)$-modules.
    
    Letting $A=\Z/\ell$, we deduce that (\ref{ss1}) is naturally isomorphic to the colimit of the Grothendieck spectral sequences
    \begin{equation}\label{ss2}
    E_2^{i,j}\coloneqq H^i(\on{Gal}(K/k), H^j(\pi(\Delta((U_{\bullet\bullet})_K)),\Z/\ell))\Rightarrow H^{i+j}(\pi(\Delta(U_{\bullet\bullet})),\Z/\ell),
    \end{equation}
    as $U_{\bullet\bullet}$ ranges over all \'etale hypercoverings of $X_{\bullet}$. By \cite[Corollary 4.6]{friedlander1982etale}, in the spectral sequence (\ref{ss2}) we may equivalently let $U_{\bullet\bullet}$ range over all rigid hypercoverings of $X_{\bullet}$. By construction, this is the spectral sequence (\ref{ss-etale-hom-quot}).
    \end{proof}

\begin{thm}\label{hochschild-serre-steenrod}
    Let $k$ be a field, let $X_{\bullet}=(X_n)_{n\geq 0}$ be a simplicial scheme, such that $X_n$ is of finite type over $k$ for all $n\geq 0$, and let $\ell$ be a prime number. On the Hochschild-Serre spectral sequence (\ref{ss-hochschild-serre-limit})
    there are, for all integers $s,i\geq 0$ and all $2\leq r\leq \infty$, homomorphisms
    \[\prescript{}{F}{\on{P}}^s\colon E_r^{a,b}\to E_r^{a,b+2s(\ell-1)},\quad \prescript{}{F}{\beta \on{P}}^s\colon E_r^{a,b}\to E_r^{a,b+2s(\ell-1)+1},\quad 0\leq 2s\leq b\]
    \[\prescript{}{B}{\on{P}}^s\colon E_r^{a,b}\to E_{t}^{a+(2s-b)(\ell-1),\ell b},\quad \prescript{}{B}{\beta \on{P}}^s\colon E_r^{a,b}\to E_{u}^{a+(2s-b)(\ell-1)+1,\ell b},\quad b\leq 2s\]
    when $\ell$ is odd, and 
    \[\prescript{}{F}{\on{Sq}}^i\colon E_r^{a,b}\to E_r^{a,b+i},\qquad 0\leq i\leq b\] \[\prescript{}{B}{\on{Sq}}^i\colon E_r^{a,b}\to E_{v}^{a+i-b,2b},\qquad b\leq i\]
    when $\ell=2$. Here we have defined
\begin{equation*}
    t\coloneqq \begin{cases}
    r-1+(2s-b)(\ell-1), & b \leq 2s< b+r-2, \\
    \ell(r-2)+1, & b+r-2\leq 2s,
    \end{cases}
\end{equation*}
\begin{equation*}
    u\coloneqq \begin{cases}
    r+(2s-b)(\ell-1), & b \leq 2s< b+r-2, \\
    \ell(r-2)+1, & b+r-2\leq 2s,
    \end{cases}
\end{equation*}
\begin{equation*}
    v\coloneqq \begin{cases}
    r+i-b, & b \leq i\leq b+r-2, \\
    2r-2, & b+r-2\leq i.
    \end{cases}
\end{equation*}
    Note that $t=u=v=2$ if $r=2$ and $t=u=v=\infty$ if $r=\infty$. 
    
    These operations satisfy the following properties. 
    
    (i) For all $2\leq r\leq \infty$, the operations on the $E_r$ page are determined by those on the $E_2$ page. More precisely, suppose that $x\in E_2^{a,b}$ survives to the $E_r$ page, let $[x]\in E_r^{a,b}$ be the image of $x$, and let $P\colon E_r^{a,b}\to E_w^{c,d}$ be one of $\prescript{}{F}{\on{P}}^s$, $\prescript{}{F}{\beta \on{P}}^s$, $\prescript{}{B}{\on{P}}^s$, $\prescript{}{B}{\beta \on{P}}^s$, $\prescript{}{F}{\on{Sq}}^i$ or $\prescript{}{B}{\on{Sq}}^i$. (The indices $c,d,w$ have been written out earlier in the statement.)  Then $P(x)$ survives to the $E_w$ page and $[P(x)]=P([x])$. 
    
    (ii) For every integer $n\geq 0$, let
    \[\{0\}=F^{n+1}H^n(X_{\bullet},\Z/\ell)\subset F^{n}H^n(X_{\bullet},\Z/\ell)\subset\cdots\subset  F^{0}H^n(X_{\bullet},\Z/\ell)=H^n(X_{\bullet},\Z/\ell)\] be the filtration of $H^n(X_{\bullet},\Z/\ell)$ induced by (\ref{ss-hochschild-serre-limit}), with projection maps
    \[\rho\colon F^aH^{a+b}(X_{\bullet},\Z/\ell)\to E_{\infty}^{a,b}.\]
    For all integers $s,i\geq 0$, write $\on{P}^s$, $\beta\on{P}^s$ and $\on{Sq}^i$ for the \'etale Steenrod operations on $H^n(X_{\bullet},\Z/\ell)$. Then, when $\ell$ is odd, we have commutative squares
    \[
    \begin{tikzcd}
    F^aH^{a+b+2s(\ell-1)}(X_{\bullet},\Z/\ell) \arrow[r,"\rho"] & E_{\infty}^{a,b+2s(\ell-1)} \\
    F^aH^{a+b}(X_{\bullet},\Z/\ell) \arrow[r,"\rho"] \arrow[u, "\on{P}^s"] & E_{\infty}^{a,b} \arrow[u, "\prescript{}{F}{\on{P}}^s"]
    \end{tikzcd}
    \]
    when $2s\leq b$, and
    \[
    \begin{tikzcd}
    F^{a+(2s-b)(\ell-1)}H^{a+b+2s(\ell-1)}(X_{\bullet},\Z/\ell) \arrow[r,"\rho"] & E_{\infty}^{a+(2s-b)(\ell-1),\ell b} \\
    F^aH^{a+b}(X_{\bullet},\Z/\ell) \arrow[r,"\rho"] \arrow[u, "\on{P}^s"] & E_{\infty}^{a,b} \arrow[u, "\prescript{}{B}{\on{P}}^s"]
    \end{tikzcd}
    \]
    when $2s\geq b$. The two squares coincide for $2s=b$. We have similar commutative squares with $\on{P}^s$ replaced by $\beta \on{P}^s$. When $\ell=2$, we have the commutative squares
    \[
    \begin{tikzcd}
    F^aH^{a+b+i}(X_{\bullet},\Z/\ell) \arrow[r,"\rho"] & E_{\infty}^{a,b+i} \\
    F^aH^{a+b}(X_{\bullet},\Z/\ell) \arrow[r,"\rho"] \arrow[u, "{\on{Sq}}^i"] & E_{\infty}^{a,b} \arrow[u, "\prescript{}{F}{\on{Sq}}^i"]
    \end{tikzcd}
    \]
    when $i\leq b$ and
    \[
    \begin{tikzcd}
    F^{a+i-b}H^{a+i+b}(X_{\bullet},\Z/\ell) \arrow[r,"\rho"] & E_{\infty}^{a+i-b, 2b} \\
    F^aH^{a+b}(X_{\bullet},\Z/\ell) \arrow[r,"\rho"] \arrow[u, "{\on{Sq}}^i"] & E_{\infty}^{a,b} \arrow[u, "\prescript{}{B}{\on{Sq}}^i"]
    \end{tikzcd}
    \]
    when $i\geq b$. The two squares coincide for $i=b$.
    
(iii) Suppose $r=2$. Then for all $s\geq 0$ the operations \[\prescript{}{F}{\on{P}}^s\colon H^a(k,H^b(\ov{X}_{\bullet},\Z/\ell))\to H^a(k,H^{b+2s(\ell-1)}(\ov{X}_{\bullet},\Z/\ell))\]
    are induced by $\on{P}^s\colon H^b(\ov{X}_{\bullet},\Z/\ell)\to H^{b+2s(\ell-1)}(\ov{X}_{\bullet},\Z/\ell)$, and similarly for $\prescript{}{F}{\beta P}^s$ and $\prescript{}{F}{\on{Sq}}^i$.
    \end{thm}

Our applications do not rely on the full strength of \Cref{hochschild-serre-steenrod}, but only on its simpler consequence \Cref{hochschild-serre-steenrod-finite-fields}. In particular, we will not need the operations with prescript $B$, nor the definitions of $t,u,v$, but we included them for completeness and possible future use.

In (ii), the fact that $\on{P}^s$, $\beta\on{P}^s$ and $\on{Sq}^i$ respect the spectral sequence filtration of $H^*(X_{\bullet},\Z/\ell)$ as described in the commutative squares is part of the statement. 

\begin{proof}
    Let $F\to E\to B$ be a (topological) Serre fibration, and consider the Serre spectral sequence
    \begin{equation}\label{top-sss}
    E_2^{i,j}\coloneqq H^i_{\on{sing}}(B,\underline{H}_{\on{sing}}^j(F,\Z/\ell))\Rightarrow H_{\on{sing}}^{i+j}(E,\Z/\ell),
    \end{equation}
    where $\underline{H}_{\on{sing}}^j(F,\Z/\ell)$ is the local system of coefficients associated to the Serre fibration.
    We will make use of the construction of (\ref{top-sss}) due to Dress \cite{dress1967spectralsequenz}; see also \cite[pp. 225-229]{mccleary2001user} or \cite[Chapter 2, \S 5]{singer2006steenrod} (for $\ell=2$). Let $\on{Sing}_{\bullet}(E)$ be the singular simplicial complex with $\Z/\ell$ coefficients of $E$, regarded as a simplicial coalgebra via the Eilenberg-Zilber map. To the Serre fibration $F\to E\to B$, Dress associated a bisimplicial $(\Z/\ell)$-coalgebra $K_{\bullet\bullet}(f)$ and an augmentation map $\lambda:K_{\bullet\bullet}(f)\to \on{Sing}_{\bullet}(E)$ such that the induced map $\lambda^*:H^*(E,\Z/\ell)\to H^*(\on{Tot}(K_{\bullet\bullet}(f)))$ is an isomorphism. The first-quadrant spectral sequence for the double complex corresponding via Dold-Kan to the bicosimplicial $(\Z/\ell)$-algebra  $\on{Hom}(K_{\bullet\bullet}(f),\Z/\ell)$ may then be identified with (\ref{top-sss}). 
    
    We now give references for the construction of Steenrod operations on (\ref{top-sss}). These references work in the setting of the spectral sequence associated to a bisimplicial $(\Z/\ell)$-coalgebra, so by Dress' construction they apply to (\ref{top-sss}). Singer \cite{singer1973steenrod} constructed Steenrod operations on (\ref{top-sss}) when $\ell=2$. Singer's definition of $\prescript{}{B}{\on{Sq}^i}$ given in \cite{singer1973steenrod} has incorrect codomain. This has been fixed in \cite[Theorem 2.15]{singer2006steenrod}, where the correct $t$ appears. When $\ell$ is odd, Mori \cite{mori1979steenrod} constructed Steenrod operations on (\ref{top-sss}), but his definition of $\prescript{}{B}{\on{P}^s}$ and $\prescript{}{B}{\beta\on{P}^s}$ has a similar problem. These errors were noticed by Sawka \cite[End of p. 741]{sawka1982odd}. In the same article, Sawka \cite{sawka1982odd} defined the operations when $\ell$ is odd. Sawka's operations are defined as homomorphisms $E_r\to E_r$ but with indeterminacy, see \cite[pp. 739-740]{sawka1982odd}. By \cite[Proposition 2.5, Proposition 6.2]{sawka1982odd}, we may view Sawka's operations as homomorphisms $E_r\to E_w$, where $w=t,u$. 
    
    We may now give references for the analogues of (i), (ii) and (iii) for (\ref{top-sss}). For (i), see \cite[Proposition 1.2]{singer1973steenrod} when $\ell=2$ and \cite[Proposition 2.2]{sawka1982odd} when $\ell$ is odd. When $\ell=2$, the proof of (ii) may be found in \cite[Proposition 1.5]{singer1973steenrod} and \cite[Theorem 2.16]{singer2006steenrod}. When $\ell$ is odd, the proof of (ii) is given in \cite[Proposition 2.4]{sawka1982odd} and \cite[Theorem 1.4 and \S 5]{mori1979steenrod}.

    We now consider (iii). When $\ell=2$, (iii) is stated in \cite[(2.77)]{singer2006steenrod}. The key tool for its proof is \cite[I, Proposition 5.1]{singer1973steenrod} (or \cite[Theorem 2.23 (2.60)]{singer2006steenrod}). An analogue of (iii) is proved for a different spectral sequence in \cite[II, Proposition 5.2]{singer1973steenrod}, but the method is general and applies to the Eilenberg-Moore spectral sequence (see \cite[II, Proposition 7.2]{singer1973steenrod}, where the author leaves the proof to the reader) and to the Serre spectral sequence \cite[II, \S 8]{singer1973steenrod}. When $\ell$ is odd, the proof is identical once we replace \cite[I, Proposition 5.1]{singer1973steenrod} by \cite[Proposition 7.1, Case 1]{sawka1982odd}.
    
    Since the spectral sequence (\ref{ss-hom-quot}) is a special case of (\ref{top-sss}), it is endowed with Steenrod operations satisfying evident analogues of (i), (ii) and (iii). Moreover, the spectral sequence (\ref{ss-hom-quot}) is naturally isomorphic to (\ref{ss-hom-quot-bgamma}), hence the properties just established for  (\ref{ss-hom-quot-bgamma}) hold for (\ref{ss-hom-quot}) as well. By the naturality of Steenrod operations, this endows the spectral sequence (\ref{ss-pro-hom-quot}) with Steenrod operations that satisfy evident analogues of (i), (ii) and (iii).
    
    By \Cref{map-of-ss}, there is a natural isomorphism of spectral sequences from (\ref{ss-hochschild-serre}) to (\ref{ss-etale-hom-quot}). We use this isomorphism to define operations on (\ref{ss-hochschild-serre}) which satisfy all the properties required in the statement of \Cref{hochschild-serre-steenrod}. The construction of (\ref{ss-hochschild-serre}) and the operations on it is functorial with respect to finite field extensions $K\subset K'$ such that $K'/k$ is Galois. Passing to the direct limit, we obtain operations with the required properties on the spectral sequence (\ref{ss-hochschild-serre-limit}).
\end{proof}

\begin{rem}
Steenrod operations on the Serre spectral sequence (\ref{top-sss}) were first defined by Araki \cite[Definition 2 p. 87 and end of p. 89]{araki1957steenrod} and independently by V\'azquez \cite{vazquez1957note}. We have not been able to access V\'azquez's paper. Araki's definition is based on Serre's construction of (\ref{top-sss}) using the cubical singular complex; see \cite{serre1951homologie}. Araki then proved a number of properties of these operations and listed them in \cite[Summary p. 89]{araki1957steenrod}. In particular, \cite[(K) p. 90]{araki1957steenrod} yields (i). It is possible to prove (ii) and (iii) for (\ref{top-sss}) using Araki's definition, but there seems to be no reference in the literature. Multiple authors affirm without proof that the operations defined using Dress' and Serre's constructions of (\ref{top-sss}) agree; see e.g. \cite[p. 328]{singer1973steenrod}, \cite[p. 737]{sawka1982odd} and \cite[\S 5]{mori1979steenrod}. This is certainly true, but we could not locate a proof in the literature. This is why we do not rely on Araki's construction for the proof of \Cref{hochschild-serre-steenrod}.
\end{rem}

Suppose now that $k=\F$ is a finite field. By \cite[Theorem 7.7]{friedlander1982etale}
(which follows from \cite[Chapitre 7, Th\'eor\`eme 1.1] {SGA4.5}), the $G$-module $H^j(\ov{X}_{\bullet},\Z/\ell)$ is finite for all $j\geq 0$. Since the cohomological dimension of $\F$ is equal to $1$, the groups $H^i(\F,H^j(\ov{X}_{\bullet},\Z/\ell))$ are trivial for all $i\geq 2$. Therefore, the $E_2$-page of (\ref{ss-hochschild-serre-limit}) induces natural surjections
\[H^j(X_{\bullet},\Z/\ell)\to H^j(\ov{X}_{\bullet},\Z/\ell)^G\]
and natural isomorphisms
\[\on{Ker}(H^j(X_{\bullet},\Z/\ell)\to H^j(\ov{X}_{\bullet},\Z/\ell))\xrightarrow{\sim}H^1(\F, H^{j-1}(\ov{X}_{\bullet},\Z/\ell)),\]
which may be rewritten as natural short exact sequences
\begin{equation}\label{hochschild-serre-cycle-map}0\to H^1(\F, H^{j-1}(\ov{X}_{\bullet},\Z/\ell))\xrightarrow{\iota} H^j(X_{\bullet},\Z/\ell)\to H^j(\ov{X}_{\bullet},\Z/\ell)^G\to 0.\end{equation}

\begin{cor}\label{hochschild-serre-steenrod-finite-fields}
Let $\F$ be a finite field, $X_{\bullet}=(X_n)_{n\geq 0}$ be a simplicial scheme such that $X_n$ is of finite type over $\F$ for all $n\geq 0$, $s\geq 0$ be an integer, $\ell$ be a prime number invertible in $\F$, and $P$ be one of $\on{P}^s$, $\beta\on{P}^s$ ($\ell$ odd), $\on{Sq}^{2s}$, $\on{Sq}^{2s+1}$ ($\ell=2$). For every $j\neq 2s$, we have a commutative diagram of short exact sequences:
\[
\adjustbox{scale=0.95,center}{
\begin{tikzcd}
0\arrow[r] & H^1(\F, H^{j-1}(\ov{X}_{\bullet},\Z/\ell)) \arrow[r] \arrow[d, "\text{$H^1(\F,P)$}"] & H^j(X_{\bullet},\Z/\ell) \arrow[r] \arrow[d, "\textrm{$P$}"] & H^j(\ov{X}_{\bullet},\Z/\ell)^G \arrow[r] \arrow[d,"\text{$P$}"] & 0 \\
0 \arrow [r] & H^1(\F, H^{j+d-1}(\ov{X}_{\bullet},\Z/\ell)) \arrow[r] & 
H^{j+d}(X_{\bullet},\Z/\ell)
\arrow[r] & H^{j+d}(\ov{X}_{\bullet},\Z/\ell)^G \arrow [r] & 0,
\end{tikzcd}
}
\]
where the rows are (\ref{hochschild-serre-cycle-map}) and $d$ is the degree of $P$.
\end{cor}

\begin{proof}
Recall that $\on{Sq}^{2s+1}=\on{Sq}^1\on{Sq}^{2s}$ when $\ell=2$. If $2s>j$, then $P$ is zero on $H^j(X_{\bullet},\Z/\ell)$ and there is nothing to prove. Suppose now that $2s<j$. By \Cref{hochschild-serre-steenrod}(ii), the map $P:H^j(X_{\bullet},\Z/\ell)\to H^{j+d}(X_{\bullet},\Z/\ell)$ respects the filtration coming from (\ref{ss-hochschild-serre-limit}), hence $P$ induces a commutative diagram as in \Cref{hochschild-serre-steenrod-finite-fields}, except that the left and right vertical maps have not been identified yet. The naturality of Steenrod operations implies that the vertical map on the right is $P$. It remains to show that the vertical map on the left is $H^1(\F,P)$. This follows from \Cref{hochschild-serre-steenrod}(iii) for $(a,b)=(1,j-1)$.
\end{proof}

In the rest of this paper, we only will apply (\ref{hochschild-serre-cycle-map}) and \Cref{hochschild-serre-steenrod-finite-fields} when $X_{\bullet}$ is the constant simplicial scheme associated to a smooth projective $\F$-variety $X$.

\subsection{Odd-degree operations and non-algebraic classes}

The material of this subsection is well known; see \cite[\S 2]{pirutka2015note}. It differs from \cite{colliot2010autour} because the ground field is finite instead of algebraically closed. The arguments of \cite{pirutka2015note} use motivic cohomology and also require the existence of a primitive $\ell^2$-root of unity. In order to make the current paper more self-contained and to avoid introducing motivic Steenrod operations, we provide alternative arguments which do not use motivic cohomology. This also saves us the trouble of proving the compatibility betweeen the \'etale Steenrod operations defined in \ref{define-steenrod} and the motivic Steenrod operations, or alternatively the \'etale Steenrod operations defined using homological algebra as in \cite{brosnan2015comparison} or \cite{guillou2019operations}.

\begin{lem}\label{projective-finite-field}
Let $\ell$ be a prime number invertible in $\F$, and let $\F$ be a finite field containing a primitive $\ell^2$-th root of unity. Then for all $n\geq 0$ the odd-degree Steenrod operations on $H^*(\P^n_{\F},\Z/\ell)$ vanish.
\end{lem}

\begin{proof}
    Viewing $\P^n_\F$ as a projective bundle over $\on{Spec}(\F)$, the projective bundle formula \cite[Expos\'e VII, Th\'eor\`eme 2.2.1]{sga5} yields a $G$-equivariant isomorphism of graded rings:
    \[\oplus_{i=0}^{n} H^{2i}(\P^n_{\ov{\F}},\mu_\ell^{\otimes i})\simeq (\Z/\ell)[h]/(h^{n+1}),\qquad |h|=2,\]
    where on the multiplication on the left side is given by cup product.
    
    Since $\F$ contains a primitive $\ell$-th root of unity, $\mu_{\ell}^{\otimes i}\simeq \Z/\ell$ for all $i$. We have  \[H^*(\F,\Z/\ell)\simeq (\Z/\ell)[\epsilon]/(\epsilon^2),\qquad |\epsilon|=1.\]
    We also denote by $\epsilon\in H^1(\P^1_{\F},\Z/\ell)$ the pullback of $\epsilon$ along the structure morphism $\P^1_{\F}\to \on{Spec}(\F)$.
    
    Consider the Hochschild-Serre spectral sequence
    \[E_2^{i,j}\coloneqq H^i(\F,H^j(\P^n_{\ov{\F}},\Z/\ell))\Rightarrow H^{i+j}(\P^n_{\F},\Z/\ell).\]
    Note that $h$ maps to the generator $\ov{h}\in E_{\infty}^{0,2}=E_2^{0,2}$ and that $\epsilon$ is the image of the generator of $\ov{\epsilon}\in E_{\infty}^{1,0}=E_2^{1,0}$. Since the Hochschild-Serre spectral sequence is multiplicative and degenerates at the $E_2$-page, we have a ring isomorphism \[E_{\infty}\simeq E_2\simeq (\Z/\ell)[\overline{h},\ov{\epsilon}]/(\ov{h}^{n+1},\ov{\epsilon}^2).\] It follows that $h^{n}\epsilon$ is not zero in the associated graded of $H^*(\P^n_{\F},\Z/\ell)$, hence $h^{n}\epsilon\neq 0$ in $H^*(\P^n_{\F},\Z/\ell)$. We obtain a graded $\Z/\ell$-algebra isomorphism
    \begin{equation}\label{proj-bundle}H^*(\P^n_{\F},\Z/\ell)\simeq (\Z/\ell)[\epsilon,h]/(\epsilon^2, h^{n+1}).\end{equation}

    We now show that the Bockstein homomorphism \[\beta\colon H^*(\P^n_{\F},\Z/\ell)\to H^*(\P^n_{\F},\Z/\ell)\] is the zero homomorphism. 
    
    To show that $\beta=0$, it suffices to show that $\beta(\epsilon)=0$ and $\beta(h)=0$. Since $\epsilon$ comes from $\on{Spec}(\F)$ and $H^2(\F,\Z/\ell)=0$, we have $\beta(\epsilon)=0$. We have a commutative triangle
    \[
    \begin{tikzcd}
        CH^1(\P^n_{\F}) \arrow[r]\arrow[dr] & H^2(\P^n_{\F},\mu_{\ell^2}) \arrow[d]  \\
        & H^2(\P^n_{\F},\mu_\ell),
    \end{tikzcd}
    \]
    where the horizontal and oblique arrows are cycle maps, and where the vertical arrow is the reduction map. Since $\F$ contains a primitive $\ell^2$-th root of unity, we may rewrite this triangle as
    \[
    \begin{tikzcd}
        CH^1(\P^n_{\F}) \arrow[r]\arrow[dr] & H^2(\P^n_{\F},\Z/\ell^2) \arrow[d]  \\
        & H^2(\P^n_{\F},\Z/\ell).
    \end{tikzcd}
    \]
    It follows that $h$ lifts to $\Z/\ell^2$, and hence $\beta(h)=0$. Thus $\beta$ vanishes on $H^*(\P^n_{\F},\Z/\ell)$, as desired. Every odd-degree Steenrod operation is of the form $\beta\on{P}^i$ (for $\ell$ odd) or $\on{Sq}^{2i+1}=\beta\on{Sq}^{2i}$ (for $\ell=2$), and hence must also vanish on $H^*(\P^n_{\F},\Z/\ell)$.
\end{proof}

\begin{thm}[\Cref{odd-vanish-intro}]\label{odd-vanish}
    Let $i\geq 1$ be an integer, $\ell\geq i$ be a prime number, $\F$ be a finite field containing a primitive $\ell^2$-th root of unity $\zeta$ and $X$ be a smooth projective $\F$-variety. Let $\alpha\in H^{2i}(X,\Z/\ell)$ be an algebraic class, that is, a class in the image of
    \[CH^i(X)/\ell\xrightarrow{\cl} H^{2i}(X,\mu_{\ell}^{\otimes i})\xrightarrow{\sim}H^{2i}(X,\Z/\ell),\]
    where the isomorphism on the right is induced by the choice of the root $\zeta^{\ell}$. Then all odd-degree Steenrod operations on $H^*(X,\Z/\ell)$ vanish on $\alpha$.
 \end{thm}
  
 \begin{proof}
    We adapt an argument given by Totaro over algebraically closed fields; see the proof of \cite[Th\'eor\`eme 2.1(1)]{colliot2010autour}. There are however some difficulties to overcome, as the $\Z/\ell$-\'etale cohomology of Grassmannians over $\F$ is not concentrated in even degrees. This is where \Cref{hochschild-serre-steenrod-finite-fields} and \Cref{projective-finite-field} come into play.
    
    Since $\ell$ is coprime with $(i-1)!$, by Jouanolou's Riemann-Roch Theorem without denominators \cite{jouanolou1970riemann}, every element of $CH^i(X)/\ell$  is a linear combination of Chern classes of vector bundles on $X$. Therefore, we may assume that $\alpha=c_i(E)$ is the $i$-th Chern class of a vector bundle $E\to X$. 
    
    If $\pi\colon \P(E)\to X$ is the projective bundle associated to $E\to X$, then $\pi^*E$ admits a rank $1$ quotient bundle and $\pi^*\colon H^*(X,\Z/\ell)\to H^*(\P(E),\Z/\ell)$ is injective. Moreover, since pullback along $\pi$ is defined on the Chow ring and is compatible with the pullback on \'etale cohomology, $\pi^*\alpha$ is still algebraic. Iterating this procedure, we may suppose that $E$ is an iterated extension of line bundles $L_j$, $j=1,\dots,r$. Then $c_i(E)$ is the degree $i$ symmetric polynomial in the $c_1(L_j)$. By the Cartan formula, it suffices to show that odd degree Steenrod operations vanish on the $c_1(L_j)$. We may thus assume that $E$ is a line bundle and that $i=1$.
    
    Let $L$ be a very ample line bundle on $X$ such that $E\otimes L$ is generated by global sections. Then 
    \[c_1(E)=c_1(E\otimes L)-c_1(L).\]
    By the linearity of the Steenrod operations, to prove that odd operations vanish on $c_i(E)$ it suffices to show that they vanish on $c_1(E\otimes L)$ and $c_1(L)$. We may thus assume that $E$ is generated by global sections, in which case $c_1(E)$ is the pullback of a cohomology class on $\P^n_{\F}$ for some integer $n\geq 1$. This reduces us to the case when $X=\P^n_{\F}$. The conclusion follows from \Cref{projective-finite-field}. 
\end{proof}

\subsection{Non-algebraic geometrically trivial classes}

Let $\sigma\in G$ be the Frobenius automorphism of $\ov{\F}$: it is a topological generator of $G$. If $M$ is a finite $G$-module, there is an isomorphism $H^1(\F,M)\simeq M/(\sigma-1)M$, sending the class of a cocycle $\{m_g\}_{g\in G}$ to the class of $m_{\sigma}$ in $M/(\sigma-1)M$. We obtain a surjection 
\begin{equation}\label{h1-surj}M\to H^1(\F,M),\end{equation} 
which is an isomorphism if and only if $G$ acts trivially on $M$; see \cite[XIII \S 1, Proposition 1]{serre1979local}.

The homomorphism (\ref{h1-surj}) is functorial in the following way. Suppose that $f\colon M\to N$ is a homomorphism of continuous $G$-modules, and let \[H^1(\F,f)\colon H^1(\F,M)\to H^1(\F,N)\] be the induced homomorphism. By definition, $H^1(\F, f)$ sends the class of a cocycle $\{m_g\}_{g\in G}$ to the class of $\{f(m_g)\}_{g\in G}$. It follows that we have a commutative diagram
\[
\begin{tikzcd}
M \arrow[r,"(\ref{h1-surj})"] \arrow[d, "f"] & H^1(\F ,M) \arrow[d, "\text{$H^1(\F, f)$}"] \arrow[r,"\sim"] & M/(\sigma-1)M \arrow[d, "f"]   \\
N \arrow[r,"(\ref{h1-surj})"] &  H^1(\F,N)  \arrow[r,"\sim"]& N/(\sigma-1)M,
\end{tikzcd}
\]
where the composition of the horizontal map is given by reduction modulo $\sigma-1$.

Before stating \Cref{h1-trick}, we recall a construction of the $\ell$-adic Hochschild-Serre spectral sequence over finite fields. Of course, this is a special case of \cite{jannsen1988continuous}. For every integer $m\geq 1$, the Hochschild-Serre spectral sequence
\[ E_2^{r,s}\coloneqq H^r(\F, H^s(\ov{X},\mu_{\ell^m}^{\otimes i}))\Rightarrow H^{r+s}(X,\mu_{\ell^m}^{\otimes i})\]
yields compatible short exact sequences
\begin{equation}\label{hochschild-serre-cycle-map-m}0\to H^1(\F, H^{2i-1}(\ov{X},\mu_{\ell^m}^{\otimes i}))\xrightarrow{\iota} H^{2i}(X,\mu_{\ell^m}^{\otimes i})\to H^{2i}(\ov{X},\mu_{\ell^m}^{\otimes i})^G\to 0.\end{equation}
Since the abelian groups $H^{2i-1}(\ov{X},\mu_{\ell^m}^{\otimes i})$ are finite, we may pass to the inverse limit in (\ref{hochschild-serre-cycle-map-m}) and obtain a short exact sequence
\begin{equation}\label{hochschild-serre-cycle-map-ell-adic} 0\to H^1(\F, H^{2i-1}(\ov{X},\Z_{\ell}(i)))\xrightarrow{\iota} H^{2i}(X,\Z_{\ell}(i))\to H^{2i}(\ov{X},\Z_{\ell}(i))^G\to 0.\end{equation}

We will produce our counterexamples by means of the following consequence of \Cref{odd-vanish}. 

\begin{prop}\label{h1-trick}
 Let $\F$ be a finite field, $\ell$ be a prime number invertible in $\F$, $1\leq i\leq \ell$ be an integer, and $X$ be a smooth projective $\F$-variety. Suppose that $\F$ contains a primitive $\ell^2$-root of unity. Let:
 
\begin{itemize}
    \item[--] $u\in H^{2i-1}(\ov{X},\Z_{\ell}(i))$,
    \item[--] $P$ be a Steenrod operation of odd degree $d$ which does not vanish on $\pi_{\ell}(u)$,
    \item[--] $u'\in H^1(\F,H^{2i-1}(\ov{X},\Z_\ell(i)))$ be the image of $u$ under the map (\ref{h1-surj}) for $M=H^{2i-1}(\ov{X},\Z_\ell(i))$,
    \item[--] $\alpha\in H^{2i}(X,\Z_\ell(i))$ be the image of $u'$ under the inclusion $\iota$ appearing in (\ref{hochschild-serre-cycle-map-ell-adic}). 
\end{itemize} 

 Then:
 
 \begin{enumerate}
     \item[$(i)$] $\alpha_{\ov{\F}}=0$, and
     \item[$(ii)$] if $G$ acts trivially on $H^{2i-1+d}(\ov{X},\Z/\ell)$, then $\alpha$ is not algebraic.
 \end{enumerate}
\end{prop}
 
 \begin{proof}
  (i) follows from the exactness of (\ref{hochschild-serre-cycle-map-ell-adic}).
  
  If $G$ acts trivially on $H^{2i-1+d}(\ov{X},\Z/\ell)$, then (\ref{h1-surj}) is an isomorphism for $M=H^{2i-1+d}(\ov{X},\Z/\ell)$. Since $P(\pi_{\ell}(u))\neq 0$ and $\pi_{\ell}(u')$ is the image of $\pi_{\ell}(u)$ under (\ref{h1-surj}) for $M=H^{2i-1}(\ov{X},\Z/\ell)$, we deduce that $H^1(\F, P)(u')\neq 0$.  Now \Cref{hochschild-serre-steenrod-finite-fields} implies that $P(\pi_\ell(\alpha))\neq 0$. Since the degree $d$ of $P$ is odd,  \Cref{odd-vanish} implies that $\pi_{\ell}(\alpha)$ is not algebraic, hence $\alpha$ is not algebraic. This proves (ii).
 \end{proof}

 \section{Proof of Theorems \ref{mainthm-tate} and \ref{mainthm}}\label{sec4}

\subsection{Proof of Theorem \ref{mainthm}}

\begin{proof}[Proof of \Cref{mainthm}]
Let $p$ be an odd prime number, $K$ be a $p$-adic field with ring of integers $R$ and residue field $\F$ is finite of characteristic $p$, and fix a field embedding $K\hookrightarrow \C$. Let $E_1$ and $E_2$ be elliptic curves over $\on{Spec}(R)$, such that $E_1$ admits a torsion point $P$ of order $4$, and that $\on{Aut}_{R}(E_2,0)$ admits an element $\sigma$ of order $4$. One can easily construct such $E_1$ and $E_2$ when $R$ is a sufficiently large finite extension of $\Z_p$. (To construct $E_1$, we start from any elliptic curve $E'/\Z_p[i]$. Any $4$-torsion point of $E'(\overline{\Q}_p)$ is defined over a finite extension $K/\Q_p$. We let $R$ be the integral closure of $\Z_p$ in $K$, we define $E_1\coloneqq (E')_R$, and we let $P$ be the closure of the $4$-torsion point. For $E_2$, consider the elliptic curve $E''/\Z_p[i]$ given by the Weierstrass equation $y^2=x^3+x$. Since $R$ contains $\Z_p[i]$, we may define $E_2\coloneqq (E'')_R$.) 

Consider the free $\Z/4$-action on  $E_1\times_R E_2$ over $R$, where the generator $1\in \Z/4$ acts on $E_1$ by translation by $P$ and on $E_2$ as $\sigma$, and define $S\coloneqq (E_1\times_{R}E_2)/(\Z/4)$. The morphism $E_1\times_{R} E_2$ given by $(u,v)\mapsto (u, v+\sigma(v))$ induces a morphism $f\colon S\to S$ which is finite \'etale of degree $2$. The morphism $f$ corresponds to the involution $\tilde{\sigma}\colon S\to S$ induced by $(u,v)\mapsto (u,\sigma(v))$. We let $\alpha\in H^1(S,\Z/2)$ be the corresponding cohomology class. 

By \cite[Lemma 5.1]{benoist2021coniveau} and Artin's comparison theorem, there exists a cohomology class $\beta'\in H^1(S_{\mathbb{C}},\Z/2)$ such that $\alpha_{\mathbb{C}}^3\beta'\neq 0$ and $(\beta')^2=0$. By the invariance of \'etale cohomology under extensions of algebraically closed fields \cite[Corollary VI.2.6]{milne1980etale}, $\beta'$ is defined over $\ov{K}$, hence over a suitable finite extension of $K'$ of $K$. Replacing $R$ by its integral closure in $K'$ and localizing at one of the maximal ideals, we may thus assume that $\beta'$ is defined over $K$. 

Since $p\neq 2$, we have an isomorphism of $R$-group schemes $\mu_{2,R}\simeq (\Z/2)_{R}$. It follows from the Kummer sequence \[1\to \mu_{2}\to \mathbb{G}_{\on{m}}\to \mathbb{G}_{\on{m}}\to 1\] over $S_R$ and $S_K$ that $H^1(S,\Z/2)=\on{Pic}(S)[2]$ and $H^1(S_{K},\Z/2)=\on{Pic}(S_{K})[2]$. Since the special fiber of $S\to\on{Spec}(R)$ is a principal divisor, the restriction map $\on{Pic}(S)\to \on{Pic}(S_K)$ is an isomorphism, hence
\begin{equation}\label{restr-surj}
    \text{The restriction map $H^1(S,\Z/2)\to H^1(S_K,\Z/2)$ is surjective.}
\end{equation}
Therefore there exists $\beta \in H^1(S,\Z/2)$ such that $\beta_{\mathbb{C}}=\beta'$. 

Let $E$ be another elliptic curve over $\on{Spec}(R)$. (For example, one may take $E=E_1$ or $E=E_2$.) Define $Y\coloneqq (S\times_RE)/(\Z/2)$, where $\Z/2$ acts on $S$ as $\tilde{\sigma}$ and on $E$ as $-\on{Id}$. Let $\pi\colon Y\to S$ be the morphism induced by the first projection, and let $Y'\to Y$ be the double cover corresponding to $\pi^*(\beta)$. Define $Z\coloneqq (Y'\times_R E)/(\Z/2)$, where $\Z/2$ acts via the involution corresponding to the double cover $Y'\to Y$ on the left and via $-\on{Id}$ on the right.  It is proved in \cite[Proof of Proposition 5.3]{benoist2021coniveau} that there exists $\sigma'\in H^3_{\on{sing}}(Z_{\mathbb{C}},\Z)[2]$ such that $\pi_2(\sigma')^2\neq 0$. By Artin's comparison theorem, we deduce the existence of $\sigma\in H^3(Z_{\mathbb{C}},\Z_2)[2]$ such that $\pi_2(\sigma)^2\neq 0$. By the invariance of \'etale cohomology under extensions of algebraically closed fields, $\sigma$ is defined over $\ov{K}$. Therefore, using (\ref{restr-surj}) and enlarging $R$ if necessary, we may suppose that $\sigma$ is defined over $R$. Letting $X$ be the special fiber of $Z$, by \Cref{smooth-proper} the restriction of $\sigma$ to $X$ satisfies the conditions of \Cref{h1-trick}, and the conclusion follows.
\end{proof}

\begin{rem}
    We give a shorter proof of the following weakening of \Cref{mainthm}: For all but finitely many primes $p$, there exist a finite field $\F$ of characteristic $p$ and a smooth projective fourfold $X$ over $\F$ such that the conclusion of \Cref{mainthm} holds for $X$.
    
    Indeed, by \cite[Proposition 5.3]{benoist2021coniveau} and Artin's comparison theorem, there exist a smooth projective fourfold $Z$ over $\mathbb{C}$ and a $2$-torsion class $\sigma\in H^3(Z,\Z_2)$ such that $\pi_2(\sigma)^2\neq 0$. By a spreading out argument, there exist a finitely generated field extension $K/\Q$, a smooth integral $\Z$-scheme of finite type $S$ with fraction field $K$, and a smooth projective scheme $\mc{Z}\to S$ such that $\mc{Z}_{\C}\simeq Z$. The morphism $S\to\on{Spec}(\Z)$ is flat, hence open. It follows that all but finitely many primes appear as residue characteristics of closed points of $S$. 
    
    Let $p\neq \ell$ be one such prime, $s\in S$ be a closed point whose residue field $k(s)$ is finite of characteristic $p$, and set $\F\coloneqq k(s)$ and  $X\coloneqq (\mc{Z})_{k(s)}$. By \cite[Proposition 7.1.9]{ega2}, there exist a discrete valuation ring $R$, with generic point $\eta$ and closed point $t$, and a morphism $\on{Spec}(R)\to S$ which sends $\eta$ to the generic point of $S$, $t$ to $s$, and such that the inclusion $K\hookrightarrow k(\eta)$ is an equality. By the invariance of \'etale cohomology under field extensions and the smooth and proper base change theorems, we have an isomorphism
    \begin{equation}\label{compat-isom}H^*(Z,\Z_2)\xrightarrow{\sim }H^*(\mc{Z}_{\ov{K}},\Z_{2})\xrightarrow{\sim}H^*(X_{\ov{k(t)}},\Z_{2})\xrightarrow{\sim} H^*(\ov{X},\Z_2),\end{equation}
    where $\ov{X}=X_{\ov{\F}}$, and a compatible isomorphism with $\Z/2$ coefficients.
    Let $\alpha\in H^3(\ov{X},\Z_{2})$ be the image of $\sigma$ under (\ref{compat-isom}).
    By \Cref{smooth-proper}, we have $\pi_2(\alpha)^2\neq 0$. Replacing $\F$ by a finite extension if necessary, we may suppose that $G$ acts trivially on $H^3(\ov{X},\Z/2)$. The conclusion now follows from \Cref{h1-trick}.
\end{rem}

We now show that \Cref{h1-trick} cannot be used to find a $3$-dimensional example $X$ with $H^3_{\nr}(X,\Q_\ell/\Z_\ell(2))\neq 0$.

\begin{prop}\label{no-dim3}
Let $\F$ be a finite field, $\ell$ be a prime number invertible in $\F$, $X$ be a smooth projective threefold over $\F$, $\alpha\in H^4(X,\Z_{\ell}(2))$ be a class such that $\alpha_{\ov{\F}}=0$, and let $\ov{\alpha}$ the reduction of $\alpha$ modulo $\ell$. Suppose that $\F$ contains a primitive $\ell^2$-th root of unity. Then all odd-degree \'etale Steenrod operations on $X$ vanish on $\ov{\alpha}$.
\end{prop}

\begin{proof}
    In view of \Cref{hochschild-serre-steenrod-finite-fields}, it suffices to show that all odd-degree \'etale Steenrod operations vanish on the image of the reduction map $H^3(\ov{X},\Z_\ell)\to H^3(\ov{X},\Z/\ell)$. This is clear if $\ell$ is odd: indeed, the only non-zero Steenrod operation of odd degree $\leq 3$ is the Bockstein, which vanishes on $\ov{\alpha}$ because $\ov{\alpha}$ lifts to an integral class. Therefore, we may suppose that $\ell=2$. In this case, the proof follows from the argument of \cite[Proposition 3.6]{benoist2021coniveau}, replacing the topological Eilenberg-Maclane spectrum $H\Z$ by the $\ell$-adic Eilenberg-Maclane spectrum $H\Z_{\ell}$ and the complex cobordism spectrum $\on{MU}$ by the $\ell$-adic cobordism spectrum $\hat{MU}$; see \cite{quick2011torsion}.
\end{proof}

\subsection{Proof of Theorem \ref{mainthm-tate}}\label{section-pfmain}

Before we begin with the proof of \Cref{mainthm-tate}, we need some observations about the $\ell$-adic cohomology and cycle class map of the classifying space of $\textrm{PGL}_{\ell}$.

\begin{lem}\label{pgl-coh}
Let $k$ be an algebraically closed field and $\ell$ be a prime number invertible in $k$.

(a) If $\ell=2$, there exists $u_3\in H^3(B_k\textrm{PGL}_2,\Z_2)$ such that $\on{Sq}^3(\pi_2(u_3))\neq 0$.

(b) If $\ell$ is odd, there exists $u_3\in H^3(B_k\textrm{PGL}_\ell,\Z_{\ell})$ such that $\beta \on{P}^1(\pi_{\ell}(u_3))\neq 0$. 
\end{lem}

\begin{proof}
By \Cref{base-change-bg}, we may suppose that $k=\mathbb{C}$ and replace \'etale cohomology by singular cohomology.

(a) We have an isomorphism $\textrm{PGL}_2(\C)\simeq \textrm{SO}_3(\C)$.
Write $H^*_{\on{sing}}(B\textrm{SO}_3(\C),\Z/2)=\F_2[w_2,w_3]$, where $w_2$ and $w_3$ are the Stiefel-Whitney classes; see e.g. \cite[Theorem 2.2, (2.3)']{toda1987cohomology}. By Wu's formula, we have $\on{Sq}^1(w_2)=w_3$. Define $u_3\coloneqq \tilde{\beta}(w_2)$. Since $\on{Sq}^1=\pi_2\circ\tilde{\beta}$, we  have $\pi_2(u_3)=w_3$, hence $\on{Sq}^3(\pi_2(u_3))=\pi_2(u_3)^2=w_3^2\neq 0$. 

(b) By \cite[Theorem 3.6(b)]{vistoli2007cohomology} or \cite[\S 3]{antieau2014topological},
    we have
    \begin{equation}\label{small-coh-pgl}
    H^i_{\on{sing}}(B\textrm{PGL}_{\ell}(\C),\Z) = \begin{cases}
  \Z  & i=0, \\
  0 & i=1,\\
  0 & i=2,\\
  \Z/\ell & i=3,\\
  \Z & i=4.\\
\end{cases}
\end{equation}
From  (\ref{small-coh-pgl}) and the universal coefficient theorem, we deduce that \[H^2_{\on{sing}}(B\textrm{PGL}_{\ell}(\C),\Z/\ell)\simeq \Z/\ell\] and that $\pi_{\ell}\colon H^3_{\on{sing}}(B\textrm{PGL}_{\ell}(\C),\Z)\to H^3_{\on{sing}}(B\textrm{PGL}_{\ell}(\C),\Z/\ell)$ is an isomorphism.

 Consider the Serre spectral sequence associated to the Serre fibration
    \[B\textrm{GL}_{\ell}(\C)\to B\textrm{PGL}_{\ell}(\C)\xrightarrow{f} K(\Z,3).\]
We have an isomorphism $H^*_{\on{sing}}(B\textrm{GL}_{\ell}(\C),\Z/\ell)\simeq (\Z/\ell)[c_1,\dots,c_{\ell}]$, where $|c_i|=2i$. Since $H^2_{\on{sing}}(B\textrm{PGL}_{\ell}(\C),\Z/\ell)\neq 0$, we have $d_2(c_1)=0$ and $d_3(c_1)=0$, hence
\[f^*\colon  H^3_{\on{sing}}(K(\Z,3),\Z/\ell)\to  H^3_{\on{sing}}(B\textrm{PGL}_{\ell}(\C),\Z/\ell)\]
is an isomorphism. By \cite[VII.4.19(2)]{mimura1991topology}, $H^3_{\on{sing}}(K(\Z,3),\Z/\ell)=(\Z/\ell)\cdot v_3$ for some $v_3\neq 0$. Let $u_3$ be the unique element of $H^3_{\on{sing}}(B\textrm{PGL}_{\ell}(\C),\Z)$ such that $u_3=f^*(v_3)$.

Recall from \cite[VII.4.19(1)]{mimura1991topology} that $H^*(K(\Z,2),\Z/{\ell})=(\Z/\ell)[z_2]$, where $|z_2|=2$. Consider the mod $\ell$ Serre spectral sequence associated to the Serre fibration
\[K(\Z,2)\to B\textrm{GL}_{\ell}(\C)\to B\textrm{PGL}_{\ell}(\C).\]
We must have $d_3(z_2)=\lambda \pi_{\ell}(u_3)$ for some $\lambda\in (\Z/\ell)^{\times}$, hence by the Kudo transgression theorem \cite[Theorem 6.14]{mccleary2001user} the class $z_2^{\ell-1}\otimes \lambda \pi_{\ell}(u_3)\in E_2^{2{\ell}-2,3}$ is transgressive and $d_{2\ell -1}(z_2^{\ell-1}\otimes \lambda \pi_{\ell}(u_3))=-\lambda \beta\on{P}^1(\pi_{\ell}(u_3))$, that is, $d_{2\ell -1}(z_2^{\ell-1}\otimes \pi_{\ell}(u_3))=- \beta\on{P}^1(\pi_{\ell}(u_3))$. Since $H^*(B\textrm{GL}_{\ell}(\C),\Z/\ell)$ is concentrated in even degrees and the total degree of $z_2^{\ell-1}\otimes \pi_{\ell}(u_3)$ is odd, $z_2^{\ell-1}\otimes \pi_{\ell}(u_3)$ does not survive in the $E_{\infty}$ page, hence $\beta\on{P}^1(\pi_{\ell}(u_3))\neq 0$, as desired.
\end{proof}

\begin{lem}\label{fulton-specialization}
Let  $k$ be a field and $\ell$ be a prime number invertible in $k$. 

(a) Let $i\geq 0$ be an integer, and suppose that $\mathcal{G}$ be a reductive group scheme over $\Z$ and that the cycle map \[CH^i(B_{\mathbb{Q}}\mathcal{G})\otimes\Z_{\ell}\to H^{2i}(B_{\ov{\mathbb{Q}}}\mathcal{G},\Z_{\ell}(i))\] is surjective. Then \[CH^i(B_k\mathcal{G})\otimes\Z_{\ell}\to H^{2i}(B_{\ov{k}}\mathcal{G},\Z_{\ell}(i))\]
is also surjective.

(b) The cycle map
\[CH^i(B_k\textrm{PGL}_{\ell})\otimes\Z_{\ell}\to H^{2i}(B_{\ov{k}}\textrm{PGL}_{\ell},\Z_{\ell}(i))\] is surjective for all $i\geq 0$.
\end{lem}

\begin{proof}
    (a) If $\on{char}(k)=0$, this follows from  the invariance of \'etale cohomology under extensions of algebraically closed fields. Suppose that $\on{char}(k)=p>0$. By the invariance of \'etale cohomology under extensions of algebraically closed fields, we may suppose that $k=\F_p$. There exist a $\mathcal{G}_{\Z_p}$-representation $V$ and an open subscheme $U\subset V$ such that $V_{\F_p}-U_{\F_p}$ and $V_{\mathbb{Q}_p}-U_{\mathbb{Q}_p}$ have codimension $\geq i+1$ in $V_{\F_p}$ and $V_{\Q_p}$, respectively, and a $\mathcal{G}_{\Z_p}$-torsor $U\to B$, where $B$ is a smooth $\Z_p$-scheme. 
    
    Let $R\coloneqq W(\ov{\F}_p)$. Fix an algebraic closure $\ov{\Q}_p$ of $\Q_p$ and an inclusion of $\Z_p$-algebras $R\subset \ov{\Q}_p$.
    We obtain a commutative diagram    
    \begin{equation}\label{rectangle-specialize}
    \begin{tikzcd}
    Z^i(B_{\Q_p})\otimes\Z_{\ell} \arrow[r,->>]  & H^{2i}(B_{\ov{\Q}_p},\Z_{\ell}(i)) \\
    Z_{\on{flat}}^i(B/\Z_p)\otimes \Z_\ell \arrow[u,->>] \arrow[r] \arrow[d]  & H^{2i}(B_R,\Z_{\ell}(i)) \arrow[d] \arrow[u]  \\ 
    Z^i(B_{\F_p})\otimes\Z_{\ell} \arrow[r] & H^{2i}(B_{\ov{\F}_p},\Z_{\ell}(i)).
    \end{tikzcd}
    \end{equation}
    Here $Z_{\on{flat}}^i(B/\Z_p)$ is the free abelian group generated by classes of integral subschemes of $B$ which are flat (that is, dominant) over $\Z_p$. The horizontal map in the middle is defined as the inverse limit in $n$ of the cycle maps 
    \[Z^i_{\on{flat}}(B/\Z_p)\to  H^{2i}(B,\mu_{\ell^n}^{\otimes i})\to H^{2i}(B_R,\mu_{\ell^n}^{\otimes i})\] of \cite[Chapitre 4, \S 2.3]{SGA4.5} (where we take $S=\on{Spec}(\Z_p)$ in the definition). The top and bottom horizontal homomorphisms are the usual cycle maps, the top vertical maps are pullbacks along the open embeddings $B_{\Q_p}\hookrightarrow B$ and $B_{\ov{\Q}_p}\hookrightarrow B_R$, and the bottom vertical maps are pullbacks along the closed embeddings $B_{\F_p}\hookrightarrow B$ and $B_{\ov{\F}_p}\hookrightarrow B_R$. The top-left vertical map is surjective: indeed, if $Z\subset B_{\Q_P}$ is an integral subscheme, its closure inside $B$ is irreducible, hence flat over $R$.
    
    Since $V_{\F_p}-U_{\F_p}$ and $V_{\mathbb{Q}_p}-U_{\mathbb{Q}_p}$ have codimension $\geq i+1$ in $V_{\F_p}$ and $V_{\Q_p}$, respectively, by definition we have $CH^i(B_{\Q_p})=CH^i(B_{\Q_p}G)$ and $CH^i(B_{\F_p})=CH^i(B_{\F_p}G)$. Moreover, the natural morphism $B\to B_{\Z_p}G$ induces a commutative diagram
    \[
    \begin{tikzcd}
    H^{2i}(B_{\ov{\Q}_p},\Z_{\ell}(i)) \arrow[r,"\sim"]  & H^{2i}(B_{\ov{\Q}_p}\mc{G},\Z_{\ell}(i)) \\
    H^{2i}(B_R,\Z_{\ell}(i)) \arrow[u] \arrow[r,"\sim"] \arrow[d]  & H^{2i}(B_R\mc{G},\Z_{\ell}(i)) \arrow[d] \arrow[u]  \\ 
    H^{2i}(B_{\ov{\F}_p},\Z_{\ell}(i)) \arrow[r,"\sim"] & H^{2i}(B_{\ov{\F}_p}\mc{G},\Z_{\ell}(i)).
    \end{tikzcd}
    \]
    By \cite[Corollary 2]{friedlander1981etale}  the two vertical maps on the right are isomorphisms. The conclusion follows from (\ref{rectangle-specialize}).
    
    (b) In view of (a), we may suppose that $k=\mathbb{Q}$. The conclusion now follows from \cite{pandharipande1998equivariant} when $\ell=2$, and is the content of \cite[Corollary 3.5]{vistoli2007cohomology} when $\ell\neq 2$.
\end{proof}

\begin{lem}\label{galois-acts-triv}
    Let $\F$ be a finite field, $\ell$ be a prime invertible in $\F$, and suppose that $\F$ contains  a primitive $\ell$-th root of unity. Then the $G$-action on the group $H^{2\ell+2}(B_{\ov{\F}}(\textrm{PGL}_{\ell}\times\Gm),\Z/\ell)$ is trivial.
\end{lem}

\begin{proof}
     Let $p\coloneqq \on{char}(\F)$. Since $\F$ contains a primitive $\ell$-th root of unity, the $G$-action on $H^{2i}(B_{\ov{\F}}\Gm,\Z/\ell)\simeq \Z/\ell$ is trivial for all $i\geq 0$. By the K\"unneth formula, it thus suffices to show that $H^{2i}(B_{\ov{\F}}\textrm{PGL}_{\ell},\Z/\ell)$ is $G$-invariant for all $0\leq i\leq \ell+1$. This is clear for $i=0$.

    By \cite[Theorem 3.6]{vistoli2007cohomology} we have $H^3(B_{\ov{\F}}\textrm{PGL}_{\ell},\Z_{\ell})=\Z/\ell$ and  
    $H^{2i+1}(B_{\ov{\F}}\textrm{PGL}_{\ell},\Z_{\ell})=0$
    for all $2\leq i\leq \ell+1$. (This is sharp, as $H^{2\ell+5}(B_{\ov{\F}}\textrm{PGL}_{\ell},\Z_{\ell})\neq 0$.) The Bockstein short exact sequence (see \cite[Lemma V.1.11]{milne1980etale})
    \begin{equation}\label{bock-2l+2}0\to H^{2i}(B_{\ov{\F}}\textrm{PGL}_{\ell},\Z_{\ell})/\ell \to H^{2i}(B_{\ov{\F}}\textrm{PGL}_{\ell},\Z/{\ell}) \to H^{2i+1}(B_{\ov{\F}}\textrm{PGL}_{\ell},\Z_{\ell})[\ell]\to 0
    \end{equation}
    is $G$-equivariant. We deduce that the reduction map 
    \[H^{2i}(B_{\ov{\F}}\textrm{PGL}_{\ell},\Z_{\ell})/\ell\to H^{2i}(B_{\ov{\F}}\textrm{PGL}_{\ell},\Z/{\ell})\]
    is an isomorphism for all $2\leq i\leq \ell+1$. Now \Cref{fulton-specialization}(b) implies that $G$ acts trivially on $H^{2i}(B_{\ov{\F}}\textrm{PGL}_{\ell},\Z/{\ell})$ for all $2\leq i\leq \ell+1$.

    Suppose now that $i=1$. Then by (\ref{small-coh-pgl}) we know that $H^2(B_{\ov{\F}}\textrm{PGL}_{\ell},\Z_{\ell})=0$, hence by (\ref{bock-2l+2}) we have a $G$-equivariant isomorphism
     \[H^2(B_{\ov{\F}}\textrm{PGL}_{\ell},\Z/{\ell})\xrightarrow{\sim} H^3(B_{\ov{\F}}\textrm{PGL}_{\ell},\Z_{\ell}).\]
     Consider the Serre spectral sequence with $\Z_\ell$ coefficients associated to the fibration $B_{\ov{\F}}\Gm\to B_{\ov{\F}}\on{GL}_\ell\to B_{\ov{\F}}\on{PGL}_\ell$:
     \[E_2^{i,j}=H^i(B_{\ov{\F}}\on{PGL}_\ell,H^j(B_{\ov{\F}}\Gm,\Z_\ell))\Rightarrow H^{i+j}(B_{\ov{\F}}\on{GL}_\ell,\Z_\ell).\]
     Since the fibration is defined over $\F$, the spectral sequence is $G$-equivariant. Let $h$ be a generator of $E_2^{0,2}=H^2(B_{\ov{\F}}\Gm,\Z_\ell)\simeq \Z_\ell(-1)$. Since $H^3(B_{\ov{\F}}\on{GL}_\ell,\Z_\ell)=0$, the group $E_2^{0,3}=H^3(B_{\ov{\F}}\on{PGL}_\ell,\Z_\ell)\simeq \Z/\ell$ does not survive to the $E_\infty$ page, and so it is generated by $d_3(h)$.
     
     Let $\sigma\in G$ be the Frobenius endomorphism. The $G$-action on $H^2(B_{\ov{\F}}\Gm,\Z_\ell)/\ell$ is trivial, hence $\sigma(h)-h$ is a multiple of $\ell h$. It follows that $\sigma(d_3(h))-d_3(h)\in H^3(B_{\ov{\F}}\on{PGL}_\ell,\Z_\ell)$ is also a multiple of $\ell$. Since $H^3(B_{\ov{\F}}\on{PGL}_\ell,\Z_\ell)= (\Z/\ell)\cdot d_3(h)$ is $\ell$-torsion, this means that $d_3(h)$ is $G$-invariant, that is, that $G$ acts trivially on $H^3(B_{\ov{\F}}\on{PGL}_\ell,\Z_\ell)$.
\end{proof}

\begin{proof}[Proof of \Cref{mainthm-tate}]
    By \cite[Remark (v) p. 5 and \S 1.1]{ekedahl2009approximating}, there exist a smooth projective $\F$-variety $X$ of dimension $2\ell+3$ and a morphism $\varphi\colon X\to B_{\F}(\textrm{PGL}_{\ell}\times \Gm)$ over $\F$ such that the pullback \[\varphi^*\colon H^*(B_{\ov{\F}}(\textrm{PGL}_{\ell}\times \Gm),\Z_\ell)\to H^*(\ov{X},\Z_\ell)\]
is an isomorphism in degrees $\leq 2\ell+2$ and is injective with torsion-free cokernel in degree $2\ell+3$. We may assume that $X$ and $\varphi$ are defined over the prime field of $X$. The cycle map
\[\cl\colon CH^*(B_{\F}\Gm)\to H^{2*}(B_{\ov{\F}}\Gm,\Z_\ell(i))\simeq \Z_{\ell}[t],\qquad |t|=2\] is an isomorphism. We have a commutative square
    \[
    \adjustbox{scale=1,center}{
    \begin{tikzcd}
    CH^*(B_{\F}\textrm{PGL}_{\ell})\otimes CH^*(B_{\F}\Gm)\otimes\Z_{\ell} \arrow[r,"\sim"] \arrow[d,"\cl\otimes \cl"] & CH^*(B_{\F}(\textrm{PGL}_{\ell}\times \Gm))\otimes\Z_{\ell} \arrow[d,"\cl"] \\
    H^{2*}(B_{\ov{\F}}\textrm{PGL}_{\ell},\Z_{\ell}(*))\otimes_{\Z_{\ell}} H^{2*}(B_{\ov{\F}}\Gm,\Z_{\ell}(*)) \arrow[r,"\sim"] & H^{2*}(B_{\ov{\F}}(\textrm{PGL}_{\ell}\times \Gm),\Z_{\ell}(*)),
    \end{tikzcd}
    }
    \]
    where the top horizontal map is the K\"unneth isomorphism of \cite[Lemma 2.12(i)]{totaro2014group}, the bottom horizontal map is the K\"unneth isomorphism in \'etale cohomology, and the vertical maps are the cycle class maps. We deduce from \Cref{fulton-specialization}(b) that the cycle map $CH^i(B_{\F}(\textrm{PGL}_{\ell}\times \Gm))\otimes\Z_{\ell} \to H^{2i}(B_{\ov{\F}}(\textrm{PGL}_{\ell}\times \Gm),\Z_{\ell}(i))$ is surjective for all integers $i\geq 0$. It now follows from the commutativity of the square
        \[
    \begin{tikzcd}
    CH^2(B_{\F}(\textrm{PGL}_{\ell}\times \Gm))\otimes\Z_{\ell} \arrow[r] \arrow[d]  & CH^2(X)\otimes\Z_{\ell} \arrow[d] \\
    H^{4}(B_{\ov{\F}}(\textrm{PGL}_{\ell}\times\Gm),\Z_{\ell}(2)) \arrow[r,"\sim"] & H^{4}(\ov{X},\Z_{\ell}(2))
    \end{tikzcd}
    \]
    that the cycle map $CH^2(X)\otimes\Z_\ell \to H^4(\ov{X},\Z_\ell(2))$ is surjective.
    
    From the $\ell$-adic Bockstein short exact sequence \cite[Lemma V.1.11]{milne1980etale} we see that
    \[\varphi^*\colon H^*(B_{\ov{\F}}(\textrm{PGL}_{\ell}\times \Gm),\Z/\ell)\to H^*(\ov{X},\Z/\ell)\]
    is an isomorphism in degrees $\leq 2\ell+2$. (Here we use the fact that the pullback is injective with torsion-free cokernel in degree $2\ell+3$.) In particular, by \Cref{galois-acts-triv}, $G$ acts trivially on $H^{2\ell+2}(\ov{X},\Z/\ell)$. Let $\varphi'\coloneqq \on{pr}_1\circ \varphi$, where \[\on{pr}_1\colon B_{\F}\textrm{PGL}_{\ell}\times B_{\F}\Gm\to B_{\F}\textrm{PGL}_{\ell}\] is the first projection. 
       As a consequence, by the K\"unneth formula in \'etale cohomology and the injectivity of $\varphi^*$ is degrees $\leq 2\ell+2$, the composition 
     \[(\varphi')^*\colon H^*(B_{\ov{\F}}\textrm{PGL}_{\ell},\Z/\ell)\to H^*(\ov{X},\Z_\ell)\]
    is injective in degrees $\leq 2\ell+2$. 
    
    Consider the Steenrod operation $P$ given by $P=\on{Sq}^3$ if $\ell=2$ and $P=\beta\on{P}^1$ if $\ell\neq 2$. The degree of $P$ is equal to the odd number $2\ell-1$. By \Cref{pgl-coh}, there exists a class $u_3\in H^3(B_{\ov{\F}}\textrm{PGL}_{\ell},\Z_{\ell}(2))$ such that $P(\pi_{\ell}(u_3))\neq 0$. Define $u\coloneqq (\varphi')^*(u_3)$. Since $P(\pi_{\ell}(u_3))$ has degree $2\ell+2$ and $(\varphi')^*$ is injective in degrees $\leq 2\ell+2$, we have $P(\pi_{\ell}(u))\neq 0$. Since $G$ acts trivially on $H^{2\ell+2}(\ov{X},\Z/\ell)$, the assumptions of \Cref{h1-trick} are satisfied, hence the conclusion follows from \Cref{h1-trick}.
\end{proof}

\section*{Acknowledgements}

We thank Burt Totaro and Jean-Louis Colliot-Th\'el\`ene for  helpful  comments and suggestions.


\begin{thebibliography}{10}

\bibitem{sga5}
{\em Cohomologie {$l$}-adique et fonctions {$L$}}.
\newblock Lecture Notes in Mathematics, Vol. 589. Springer-Verlag, Berlin-New
  York, 1977.
\newblock S\'{e}minaire de G\'{e}ometrie Alg\'{e}brique du Bois-Marie
  1965--1966 (SGA 5), Edit\'{e} par Luc Illusie.

\bibitem{antieau2016tate}
Benjamin Antieau.
\newblock On the integral {T}ate conjecture for finite fields and
  representation theory.
\newblock {\em Algebr. Geom.}, 3(2):138--149, 2016.

\bibitem{antieau2014topological}
Benjamin Antieau and Ben Williams.
\newblock The topological period-index problem over 6-complexes.
\newblock {\em J. Topol.}, 7(3):617--640, 2014.

\bibitem{araki1957steenrod}
Sh\^{o}r\^{o} Araki.
\newblock Steenrod reduced powers in the spectral sequences associated with a
  fibering. {I}, {II}.
\newblock {\em Mem. Fac. Sci. Ky\={u}sy\={u} Univ. A}, 11:15--64, 81--97, 1957.

\bibitem{atiyah1962analytic}
M. F. Atiyah and F. Hirzebruch, Analytic cycles on complex manifolds. {\em Topology} 1 (1962), 25--45.

\bibitem{benoist2021coniveau}
Olivier Benoist and John~Christian Ottem.
\newblock Two coniveau filtrations.
\newblock {\em Duke Math. J.}, 170(12):2719--2753, 2021.

\bibitem{brosnan2015comparison}
Patrick Brosnan and Roy Joshua.
\newblock Comparison of motivic and simplicial operations in mod-{$l$}-motivic
  and \'{e}tale cohomology.
\newblock In {\em Feynman amplitudes, periods and motives}, volume 648 of {\em
  Contemp. Math.}, pages 29--55. Amer. Math. Soc., Providence, RI, 2015.

\bibitem{colliot1999conjectures}
Jean-Louis Colliot-Th\'{e}l\`ene.
\newblock Conjectures de type local-global sur l'image des groupes de {C}how dans la cohomologie \'{e}tale.
\newblock In {\em Algebraic {$K$}-theory} ({S}eattle, {WA}, 1997), Proc. Sympos. Pure. Math., volume 67, pages 1--12, Amer. Math. Soc., Providence, RI, 1999.

\bibitem{colliot2023liste} 
Jean-Louis Colliot-Th\'{e}l\`ene.
\newblock Une liste de probl\`emes.
\newblock {\em arXiv preprint arXiv:2212.05791}, to appear in {\em Mathematics Going Forward, Collected Mathematical Brushstrokes}, LNM 2313, Springer (2023). 

\bibitem{colliot2013cycles}
Jean-Louis Colliot-Th\'{e}l\`ene and Bruno Kahn.
\newblock Cycles de codimension 2 et {$H^3$} non ramifi\'{e} pour les
  vari\'{e}t\'{e}s sur les corps finis.
\newblock {\em J. K-Theory}, 11(1):1--53, 2013.

\bibitem{colliot2020conjecture}
Jean-Louis Colliot-Th{\'e}l{\`e}ne and Federico Scavia.
\newblock Sur la conjecture de {T}ate enti\`ere pour le produit d'une courbe et
  d'une surface {$CH_0$}-triviale sur un corps fini.
\newblock {\em arXiv preprint arXiv:2001.10515}, 2020.

\bibitem{colliot2010autour}
Jean-Louis Colliot-Th\'{e}l\`ene and Tam\'{a}s Szamuely.
\newblock Autour de la conjecture de {T}ate \`a coefficients {${\bf Z}_{\ell}$}
  pour les vari\'{e}t\'{e}s sur les corps finis.
\newblock In {\em The geometry of algebraic cycles}, volume~9 of {\em Clay
  Math. Proc.}, pages 83--98. Amer. Math. Soc., Providence, RI, 2010.

\bibitem{SGA4.5}
Pierre Deligne.
\newblock Cohomologie \'etale.
\newblock S\'eminaire de g\'eométrie alg\'ebrique du Bois-Marie SGA $4\frac{1}{2}$. \emph{Lecture Notes in Mathematics}, 569. Springer-Verlag, Berlin, 1977.

\bibitem{dress1967spectralsequenz}
A.~Dress.
\newblock Zur {S}pectralsequenz von {F}aserungen.
\newblock {\em Invent. Math.}, 3:172--178, 1967.

\bibitem{ekedahl2009approximating}
Torsten Ekedahl.
\newblock Approximating classifying spaces by smooth projective varieties.
\newblock {\em arXiv preprint arXiv:0905.1538}, 2009.

\bibitem{epstein1966steenrod}
D.~B.~A. Epstein.
\newblock Steenrod operations in homological algebra.
\newblock {\em Invent. Math.}, 1:152--208, 1966.

\bibitem{feng2020etale}
Tony Feng.
\newblock \'{E}tale {S}teenrod operations and the {A}rtin-{T}ate pairing.
\newblock {\em Compos. Math.}, 156(7):1476--1515, 2020.

\bibitem{friedlander1973fibrations}
Eric~M. Friedlander.
\newblock Fibrations in etale homotopy theory.
\newblock {\em Inst. Hautes \'{E}tudes Sci. Publ. Math.}, (42):5--46, 1973.

\bibitem{friedlander1981etale}
Eric M. Friedlander, and Brian Parshall. 
\newblock Étale cohomology of reductive groups. 
\newblock {\em Algebraic $K$-theory, Evanston 1980} (Proc. Conf., Northwestern Univ., Evanston, Ill., 1980), pp. 127--140, Lecture Notes in Math., 854, Springer, Berlin, 1981.

\bibitem{friedlander1982etale}
Eric~M. Friedlander.
\newblock {\em \'{E}tale homotopy of simplicial schemes}, volume 104 of {\em
  Annals of Mathematics Studies}.
\newblock Princeton University Press, Princeton, N.J.; University of Tokyo
  Press, Tokyo, 1982.

\bibitem{goerss2009simplicial}
Paul~G. Goerss and John~F. Jardine.
\newblock {\em Simplicial homotopy theory}.
\newblock Modern Birkh\"{a}user Classics. Birkh\"{a}user Verlag, Basel, 2009.
\newblock Reprint of the 1999 edition [MR1711612].

\bibitem{ega2}
Alexander Grothendieck and Jean Dieudonn{\'e}.
\newblock {\em {\'E}l{\'e}ments de g{\'e}om{\'e}trie alg{\'e}brique {II}},
  volume~8 of {\em Publications {M}ath{\'e}matiques}.
\newblock Institute des {H}autes {\'E}tudes {S}cientifiques., 1961.

\bibitem{guillou2019operations}
Bert Guillou and Chuck Weibel.
\newblock Operations in \'{e}tale and motivic cohomology.
\newblock {\em Trans. Amer. Math. Soc.}, 372(2):1057--1090, 2019.

\bibitem{jannsen1988continuous}
Uwe Jannsen.
\newblock Continuous \'{e}tale cohomology.
\newblock {\em Math. Ann.}, 280(2):207--245, 1988.

\bibitem{jardine1989}
J.~F. Jardine.
\newblock Steenrod operations in the cohomology of simplicial presheaves.
\newblock In {\em Algebraic {$K$}-theory: connections with geometry and
  topology ({L}ake {L}ouise, {AB}, 1987)}, volume 279 of {\em NATO Adv. Sci.
  Inst. Ser. C: Math. Phys. Sci.}, pages 103--116. Kluwer Acad. Publ.,
  Dordrecht, 1989.

\bibitem{jouanolou1970riemann}
J.~P. Jouanolou.
\newblock Riemann-{R}och sans d\'{e}nominateurs.
\newblock {\em Invent. Math.}, 11:15--26, 1970.

\bibitem{kahn2012classes}
Bruno Kahn.
\newblock Classes de cycles motiviques \'{e}tales.
\newblock {\em Algebra Number Theory}, 6(7):1369--1407, 2012.

\bibitem{kameko2015tate}
Masaki Kameko.
\newblock On the integral {T}ate conjecture over finite fields.
\newblock {\em Math. Proc. Cambridge Philos. Soc.}, 158(3):531--546, 2015.

\bibitem{kollar1992trento}
J\'{a}nos Koll\'{a}r.
\newblock Trento examples.
\newblock In {\em Classification of irregular varieties ({T}rento, 1990)},
  volume 1515 of {\em Lecture Notes in Math.}, pages 134--139. Springer,
  Berlin, 1992.


\bibitem{mccleary2001user}
John McCleary.
\newblock {\em A user's guide to spectral sequences}, volume~58 of {\em
  Cambridge Studies in Advanced Mathematics}.
\newblock Cambridge University Press, Cambridge, second edition, 2001.

\bibitem{milne1980etale}
James~S. Milne.
\newblock {\em \'{E}tale cohomology}, volume~33 of {\em Princeton Mathematical
  Series}.
\newblock Princeton University Press, Princeton, N.J., 1980.

\bibitem{mimura1991topology}
Mamoru Mimura and Hirosi Toda.
\newblock {\em Topology of {L}ie groups. {I}, {II}}, volume~91 of {\em
  Translations of Mathematical Monographs}.
\newblock American Mathematical Society, Providence, RI, 1991.
\newblock Translated from the 1978 Japanese edition by the authors.

\bibitem{mori1979steenrod}
Masamitsu Mori.
\newblock The {S}teenrod operations in the {E}ilenberg-{M}oore spectral
  sequence.
\newblock {\em Hiroshima Math. J.}, 9(1):17--34, 1979.

\bibitem{pandharipande1998equivariant}
Rahul Pandharipande.
\newblock Equivariant {C}how rings of {${\rm O}(k),\ {\rm SO}(2k+1)$}, and
  {${\rm SO}(4)$}.
\newblock {\em J. Reine Angew. Math.}, 496:131--148, 1998.

\bibitem{pirutka2011groupe}
Alena Pirutka.
\newblock Sur le groupe de {C}how de codimension deux des vari\'{e}t\'{e}s sur
  les corps finis.
\newblock {\em Algebra Number Theory}, 5(6):803--817, 2011.

\bibitem{pirutka2015note}
Alena Pirutka and Nobuaki Yagita.
\newblock Note on the counterexamples for the integral {T}ate conjecture over
  finite fields.
\newblock {\em Doc. Math.}, (Extra vol.: Alexander S. Merkurjev's sixtieth
  birthday):501--511, 2015.

\bibitem{quick2011torsion}
Gereon Quick.
\newblock Torsion algebraic cycles and \'{e}tale cobordism.
\newblock {\em Adv. Math.}, 227(2):962--985, 2011.

\bibitem{saito1989observations} 
Shuji Saito.
\newblock Some observations on motivic cohomology of arithmetic schemes.
\newblock {\em Invent. Math.} 98(1989), no.2, 371--404.

\bibitem{sawka1982odd}
John Sawka.
\newblock Odd primary {S}teenrod operations in first-quadrant spectral
  sequences.
\newblock {\em Trans. Amer. Math. Soc.}, 273(2):737--752, 1982.

\bibitem{scavia2022autour}
Federico Scavia.
\newblock {Autour de la conjecture de Tate enti\`ere pour certains produits de
  dimension $3$ sur un corps fini}.
\newblock {\em {Épijournal de Géométrie Algébrique}}, {Volume 6}, 2022, article no. 11.

\bibitem{schoen1998integral}
Chad Schoen.
\newblock An integral analog of the {T}ate conjecture for one-dimensional
  cycles on varieties over finite fields.
\newblock {\em Math. Ann.}, 311(3):493--500, 1998.


\bibitem{serre1951homologie}
Jean-Pierre Serre.
\newblock Homologie singuli\`ere des espaces fibr\'{e}s. {A}pplications.
\newblock {\em Ann. of Math. (2)}, 54:425--505, 1951.

\bibitem{serre1979local}
Jean-Pierre Serre.
\newblock {\em Local fields}, volume~67 of {\em Graduate Texts in Mathematics}.
\newblock Springer-Verlag, New York-Berlin, 1979.
\newblock Translated from the French by Marvin Jay Greenberg.

\bibitem{singer1973steenrod}
William~M. Singer.
\newblock Steenrod squares in spectral sequences. {I}, {II}.
\newblock {\em Trans. Amer. Math. Soc.}, 175:327--336; ibid. 175 (1973),
  337--353, 1973.

\bibitem{singer2006steenrod}
William~M. Singer.
\newblock {\em Steenrod squares in spectral sequences}, volume 129 of {\em
  Mathematical Surveys and Monographs}.
\newblock American Mathematical Society, Providence, RI, 2006.

\bibitem{stacks-project}
The {Stacks Project Authors}.
\newblock \textit{Stacks Project}.
\newblock http://stacks.math.columbia.edu.


\bibitem{toda1987cohomology}
Hiroshi Toda.
\newblock Cohomology of classifying spaces.
\newblock In {\em Homotopy theory and related topics ({K}yoto, 1984)}, volume~9
  of {\em Adv. Stud. Pure Math.}, pages 75--108. North-Holland, Amsterdam,
  1987.

\bibitem{totaro2014group}
Burt Totaro.
\newblock {\em Group Cohomology and Algebraic Cycles}.
\newblock Cambridge Tracts in Mathematics. Cambridge University Press, 2014.

\bibitem{vazquez1957note}
Roberto V\'{a}zquez.
\newblock Note on {S}teenrod squares in the spectral sequence of a fibre space.
\newblock {\em Bol. Soc. Mat. Mexicana (2)}, 2:1--8, 1957.

\bibitem{vezzosi2000chow}
Gabriele Vezzosi.
\newblock On the Chow ring of the classifying stack of ${\rm PGL}_{3,\C}$.
\newblock {\em J. Reine Angew. Math.} 523 (2000), 1--54.

\bibitem{vistoli2007cohomology}
Angelo Vistoli.
\newblock On the cohomology and the {C}how ring of the classifying space of
  {${\rm PGL}_{p}$}.
\newblock {\em Journal f{\"u}r die reine und angewandte Mathematik (Crelles
  Journal)}, 2007(610):181--227, 2007.


\end{thebibliography}
\end{document}